\documentclass{article}

\usepackage{booktabs}       
\usepackage{amsfonts}       
\usepackage{nicefrac}       
\usepackage{microtype}      

\usepackage{fncylab,enumerate,wrapfig,placeins}
\usepackage{bbm}
\usepackage{graphicx}
\usepackage{enumitem}  
\usepackage{amsmath,amsthm,amsfonts,amssymb}  
\usepackage{pifont}

\usepackage{geometry}

\usepackage[usenames,dvipsnames]{xcolor}
\usepackage[colorlinks = true, 
linkcolor = RoyalPurple,
urlcolor  = CadetBlue,
citecolor = blue,
anchorcolor = blue]{hyperref}

\pdfoutput=1  

\usepackage{hyperref}       
\usepackage{url}            

\usepackage{graphicx}  
\usepackage{relsize} 
\usepackage{accents}  

\usepackage{caption}
\usepackage{textgreek,upgreek,bm}

\usepackage{titlesec}

\newtheorem{theorem}{Theorem}[section]
\newtheorem{lemma}[theorem]{Lemma} 
\newtheorem{proposition}[theorem]{Proposition}

\newlength{\noteWidth}
\long\def\notes#1{\ifinner
	{\footnotesize #1}
	\else 
	\marginpar{\parbox[t]{\noteWidth}{\raggedright\tiny#1}}  
	\fi\typeout{#1}}
	
\def\Bss{\textrm{g}}
\def\Sens{\mathcal{S}}

\usepackage{cleveref}

\Crefname{corollary}{Corollary}{Corollaries}
\Crefname{eqnarray}{eq.}{eqs.}
\Crefname{equation}{eq.}{eqs.}

\Crefname{figure}{Fig.}{Figs.}
\Crefname{tabular}{Tab.}{Tabs.}
\Crefname{table}{Tab.}{Tabs.}
\Crefname{proposition}{Prop.}{Propositions}
\Crefname{theorem}{Thm.}{Thms.}
\Crefname{definition}{Def.}{Defs.} 
\Crefname{section}{Section}{Sections}
\Crefname{lemma}{Lemma}{Lemmas}
\Crefname{assumption}{Assumption}{Assumptions}

\def\urls#1{{\footnotesize\url{#1}}}

 \def\thetaopt{\theta^{\text{\tiny\sf opt}}}
\def\mindex#1{\index{#1}}

\def\baru{\bar{u}}

\DeclareFontFamily{U}{mathx}{\hyphenchar\font45}
\DeclareFontShape{U}{mathx}{m}{n}{<-> mathx10}{}
\DeclareSymbolFont{mathx}{U}{mathx}{m}{n}
\DeclareMathAccent{\widebar}{0}{mathx}{"73}

\def\barUpupsilon{\widebar{\Upupsilon}}

\def\tilnabla{{\widetilde{\nabla}\!}}
 \def\haUpupsilon{\widehat{\Upupsilon}}

\newcommand{\qsaprobe}{{\scalebox{1.1}{$\upxi$}}}  
\newcommand{\bfqsaprobe}{{\scalebox{1.1}{$\bm{\upxi}$}}}  
\def\Obj{\Upgamma}  

\def\LambdaF{\Lambda^{\text{\tiny\sf F}}}

\def\ODEstate{\Uptheta} 
\def\bfODEstate{\bm{\Uptheta}}
\def\bfLambda{\bm{\Lambda}}

\def\haODEstate{\widehat{\Uptheta}}
\def\haDelta{\widehat{\Delta}}

\def\odestate{\upvartheta}

\newcommand{\bbblot}{\raise1pt\hbox{\vrule height .4ex width .4ex depth .05ex}}

\long\def\defbox#1{\framebox[.9\hsize][c]{\parbox{.85\hsize}{%
	\parindent=0pt
	\baselineskip=12pt plus .1pt      
	\parskip=6pt plus 1.5pt minus 1pt 
	#1}}}

\long\def\beginbox#1\endbox{\subsection*{}%
\hbox{\hspace{.05\hsize}\defbox{\medskip#1\bigskip}}%
\subsection*{}}

\def\endbox{}

\def\archival#1{} 

\def\FRAC#1#2#3{\genfrac{}{}{}{#1}{#2}{#3}}

\def\ddt{{\mathchoice{\FRAC{1}{d}{dt}}%
{\FRAC{1}{d}{dt}}%
{\FRAC{3}{d}{dt}}%
{\FRAC{3}{d}{dt}}}}

\def\ddr{{\mathchoice{\FRAC{1}{d}{dr}}%
{\FRAC{1}{d}{dr}}%
{\FRAC{3}{d}{dr}}%
{\FRAC{3}{d}{dr}}}}

\def\ddtp{{\mathchoice{\FRAC{1}{d^{\hbox to 2pt{\rm\tiny +\hss}}}{dt}}%
{\FRAC{1}{d^{\hbox to 2pt{\rm\tiny +\hss}}}{dt}}%
{\FRAC{3}{d^{\hbox to 2pt{\rm\tiny +\hss}}}{dt}}%
{\FRAC{3}{d^{\hbox to 2pt{\rm\tiny +\hss}}}{dt}}}}

\def\ddyp{{\mathchoice{\FRAC{1}{d^{\hbox to 2pt{\rm\tiny +\hss}}}{dy}}%
{\FRAC{1}{d^{\hbox to 2pt{\rm\tiny +\hss}}}{dy}}%
{\FRAC{3}{d^{\hbox to 2pt{\rm\tiny +\hss}}}{dy}}%
{\FRAC{3}{d^{\hbox to 2pt{\rm\tiny +\hss}}}{dy}}}}

\def\half{{\mathchoice{\FRAC{1}{1}{2}}%
{\FRAC{1}{1}{2}}%
{\FRAC{3}{1}{2}}%
{\FRAC{3}{1}{2}}}}

\def\limsup{\mathop{\rm lim{\,}sup}}

\def\argmin{\mathop{\rm arg{\,}min}}

\def\bfX{{\bf X}}

\def\bfmath#1{{\mathchoice{\mbox{\boldmath$#1$}}%
{\mbox{\boldmath$#1$}}%
{\mbox{\boldmath$\scriptstyle#1$}}%
{\mbox{\boldmath$\scriptscriptstyle#1$}}}}

\def\bfPhi{\bfmath{\Phi}}
\def\bfPsi{\bfmath{\Psi}}

\def\haodestate{\widehat{\upvartheta}} 
\def\haqsaprobe{\widehat{\qsaprobe}}

\def\bfmX{\bfmath{X}}

\def\bfmY{\bfmath{Y}}

\def\bfmhhaY{\bfmath{\hhaY}} 
\def\bfmhhaY{\hbox to 0pt{$\widehat{\bfmY}$\hss}\widehat{\phantom{\raise 1.25pt\hbox{$\bfmY$}}}}

\def\haf{{\hat f}}
\def\hag{{\hat g}}
\def\hah{{\hat h}}

\def\hau{{\hat u}}

\def\haA{\widehat A}

\def\tilg{\tilde g}

\def\tilu{\tilde u}

\def\tilF{\widetilde{F}}

\def\clE{{\cal E}}

\def\clL{{\cal L}}

\def\clW{{\cal W}}

\def\clY{{\cal Y}}

\def\clL{{\cal L}}

\def\eqdef{\mathbin{:=}}

\def\Expect{{\sf E}}

\def\epsy{\varepsilon}

\def\formtmp#1#2{{\vskip12pt\noindent\fboxsep=0pt\colorbox{#1}{\vbox{\vskip3pt\hbox to \textwidth{\hskip3pt\vbox{\raggedright\noindent\textbf{#2\vphantom{Qy}}}\hfill}\vspace*{3pt}}}\par\vskip2pt%
\noindent\kern0pt}}

\def\barf{{\widebar{f}}}
\def\barg{{\widebar{g}}}

\def\barh{{\overline {h}}}

\def\barA{{\bar{A}}}

\def\barF{{\bar{F}}}

{\end{list}}

\def\ass(#1:#2){(#1\ref{#1:#2})}

\def\ritem#1{
\item[{\sf \ass(\current_model:#1)}]
}

\newenvironment{recall-ass}[1]{%
\begin{description}
\def\current_model{#1}}{
\end{description}
}

\def\sq{\hbox{\rlap{$\sqcap$}$\sqcup$}}
\def\qed{\ifmmode\sq\else{\unskip\nobreak\hfil
\penalty50\hskip1em\null\nobreak\hfil\sq
\parfillskip=0pt\finalhyphendemerits=0\endgraf}\fi}

\newcommand{\blot}{\vrule height 1.1ex width .9ex depth -.1ex }
\def\qedb{\ifmmode\blot\else{\vspace{-.2cm}\unskip\nobreak\hfil
\penalty50\hskip1em\null\nobreak\hfil\blot
\parfillskip=0pt\finalhyphendemerits=0\endgraf}\fi}

\newcounter{rmnum}

\newcounter{anum}

\newcommand{\field}[1]{\mathbb{#1}}

\def\Re{\field{R}}

\def\Co{\field{C}}

\def\Expect{{\sf E}}

\def\barAss{\bar{A}_{\sf ss}}
\def\barAsf{\bar{A}_{\sf sf}}
\def\barAfs{\bar{A}_{\sf fs}}
\def\barAff{\bar{A}_{\sf ff}}

\def\Upupsilonss{\Upupsilon^{\sf ss}}
\def\Upupsilonsf{\Upupsilon^{\sf sf}}
\def\Upupsilonfs{\Upupsilon^{\sf fs}}
\def\Upupsilonff{\Upupsilon^{\sf ff}}

\def\barUpupsilonff{\barUpupsilon^{\sf ff}}

 \def\haUpupsilonff{\widehat{\Upupsilon}^{\sf ff}}

\def\clD{\mathcal{D}}

\def\Dh{\clD^h}
\def\Dg{\clD^g}

\def\transpose{{\intercal}}

\def\Real{\text{\rm Re\,}}

\def\argmin{\mathop{\rm arg\, min}}

\def\epsy{\varepsilon}

\def\haF{\widehat{F}}

\def\haY{\widehat{Y}}

\def\hhaY{\hbox to 0pt{$\haY$\hss}\widehat{\phantom{\raise 1.25pt\hbox{Y}}}}

\def\haA{\widehat A}

\def\haY{\widehat Y}

\def\bfPhi{\bfmath{\Phi}}

\def\hahaf{\hat{\hat{f}}}

\def\hahah{\hat{\hat{h}}}

\def\barUpupsilon{\widebar{\Upupsilon}}

\def\tilXi{\widetilde{\Xi}}

\def\prstate{\Upomega}

\def\tilnabla{{\widetilde{\nabla}\!}}

\newsavebox{\junk}
\savebox{\junk}[1.6mm]{\hbox{$|\!|\!|$}}

\def\limsup{\mathop{\rm lim\ sup}}

\def\argmin{\mathop{\rm arg\, min}}

\def\clD{{\cal D}}
\def\clE{{\cal E}}

\def\clL{{\cal L}}

\def\clT{{\cal T}}

\def\clW{{\cal W}}

\def\clY{{\cal Y}}

\makeatletter
\newcommand\gobblepars{%
\@ifnextchar\par%
{\expandafter\gobblepars\@gobble}%
{}}
\makeatother

\def\whamrm#1{\smallbreak\pagebreak[3]%
\noindent\text{\rm#1}\ \ \gobblepars}

\def\whamit#1{\smallbreak\pagebreak[3]%
\noindent\textit{#1}\ \ \gobblepars}

\def\wham#1{\smallbreak\pagebreak[3]%
\noindent\textbf{#1}\ \ \gobblepars}

\def\bdd#1{b^{\text{\rm\tiny\ref{#1}}}}
\def\bdde#1{\varrho^{\text{\rm\tiny\ref{#1}}}}

\def\fasttarg{\uplambda^{\!*}} 
\def\Lipf{L_f} 
\def\Lipfasttarg{L_\uplambda} 
\def\bivstate{\textup{\textsf{Y}}}

\def\BIGstate{X}
\def\imp{\text{\rm g}}

\def\clEF{\clE^{\text{\tiny\sf F}}}
\def\clWF{\clW^{\text{\tiny\sf F}}}

\def\tilLambdaF{{\widetilde\Lambda}^{\text{\tiny\sf F}}}

\def\clYn{\mathcal{Y}^{\text{\tiny\sf n}}}  
\def\ctr{\text{\tiny{\sf ctr}}}

\def\barAss{\bar{A}^{\sf s}}
\def\barAsf{\bar{A}^{\sf sf}}
\def\barAfs{\bar{A}^{\sf fs}}
\def\barAff{\bar{A}^{\sf f}}

\def\barAsstar{\bar{A}^{{\sf s}*}}

\def\clYn{\mathcal{Y}^{\text{\tiny\sf n}}}  
\def\ctr{\text{\tiny{\sf ctr}}}

\def\rF{\text{\rm F}}
\def\rG{\text{\rm G}}
\def\rH{\text{\rm H}}
\def\rJ{\text{\rm J}}
\def\rM{\text{\rm M}}

\def\chclYn{\check{\mathcal{Y}}^{\text{\tiny\sf n}}}
\def\cqsaprobe{\check{\qsaprobe}}

\def\witem{\wham{\small$\triangle$}}

\def\whamit#1{\smallbreak\pagebreak[3]%
	\noindent\textit{#1}\ \ \gobblepars}

\def\wham#1{\smallbreak\pagebreak[3]%
	\noindent\textbf{#1}\ \ \gobblepars}
	
\def\whamrm#1{\smallbreak\pagebreak[3]%
	\noindent{{\upshape\rm#1}}\ \ \gobblepars}

\def\witem{\wham{\small$\triangle$}}

\title{Markovian Foundations for 
	\\
	Quasi-Stochastic Approximation in Two Timescales
	\\[0.7em]
	\Large Extended Version}

\author{Caio Kalil Lauand and Sean Meyn
	\thanks{Financial support from ARO award W911NF2010055     and NSF award    CCF~2306023   
   is gratefully acknowledged.} 
	\thanks{Caio Kalil Lauand and Sean Meyn are with the Department of Electrical and Computer Engineering, University of Florida, Gainesville, FL, USA.
Emails: {\tt\small caio.kalillauand@ufl.edu} and {\tt\small meyn@ece.ufl.edu}}%
}

\begin{document}

\maketitle

\begin{abstract}
Many machine learning and optimization algorithms can be cast as instances of stochastic approximation (SA). The convergence rate of these algorithms is known to be slow, with the optimal mean squared error (MSE) of order $O(n^{-1})$. In prior work it was shown that MSE bounds approaching $O(n^{-4})$ can be achieved through the framework of quasi-stochastic approximation (QSA); essentially SA with careful choice of deterministic exploration. These results are extended to two time-scale algorithms, as found in policy gradient methods of reinforcement learning and extremum seeking control.  The extensions are made possible in part by a new approach to analysis,   allowing for the interpretation of two timescale algorithms as instances of single timescale QSA,
made possible by the theory of negative Lyapunov exponents for QSA.
The general theory is illustrated with applications to extremum seeking control (ESC).

\end{abstract}

\clearpage

\tableofcontents

\clearpage

\section{Introduction}

Stochastic approximation (SA) was born in the pioneering work of Robbins and Monro in the 1950s \cite{robmon51a} and remains a significant topic for research, particularly in the machine learning and optimization communities. Early application to reinforcement learning may be found in \cite{konbor99,tsi94a,royThesis98};
see \cite{bacmou11,moujunwaibarjor20,orvkerprobacluc22}  for more recent advances for applications to machine learning.   

Any SA algorithm is designed to solve a root-finding problem $\barf(\theta^*) = 0$,  in which $\barf: \Re^d \to \Re^d$ may be expressed $\barf(\theta) =   \Expect [f(\theta,\Phi)] $ for $\theta\in\Re^d$ with $\Phi$ a random variable.   
The algorithm of Robbins and Monro obtains estimates   through the recursion,
\begin{equation}
	\theta_{n+1} = \theta_n + \alpha_{n+1} f(\theta_n, \Phi_{n+1}) 
	\, , \quad
	n \geq 0
	\label{e:SAGen}
\end{equation}
where $\{\Phi_{n}\}$ is a sequence of random vectors converging to $\Phi$ (in distribution) and $\{\alpha_n\}$ is a step-size sequence.

The recursion is designed to mimic the  \textit{mean flow}, defined as the ODE $\dot{\odestate}_t = \barf(\odestate_t)$.  Analysis of \eqref{e:SAGen} proceeds by comparing parameter estimates with solutions to the mean flow \cite{bor20a}.   
Common choices for   $\{\alpha_n\}$ include (i) vanishing step-size sequences of the form $\alpha_n = n^{-\rho}, \rho \in (1/2,1]$; and (ii) constant step-sizes, in which $\alpha_n \equiv \alpha>0$.   Under general conditions on  $\barf$ and $\bfPhi$, the choice (i) leads to almost sure convergence of $\{\theta_n\}$ to $\theta^*$ for each initial condition $\theta_0 \in \Re^d$.   For (ii), there is little hope for convergence,  though convergence typically holds for the averaged parameters \cite{bacmou11,bormey00a,laumey23a,bor20a}. 

Without averaging,  either choice of step-size typically results in MSE bounds of  order $\Expect[\| \theta_n - \theta^*\|^2]= O(\alpha_n)$,  and this can be improved to  $O(n^{-1})$ for vanishing step-sizes with $\rho<1$, and even for constant step-size algorithms in special cases  \cite{bacmou11,moujunwaibarjor20,laumey23a}.

Recent results have established much better MSE bounds when $\bfPhi$ is deterministic. In particular, it has been shown in \cite{laumey22e} that MSE rates of order $O(\alpha_n^2)$ can be achieved when $\bfPhi$ is a mixture of sinusoids with carefully selected frequencies. With averaging, these rates are sped up to $O(\alpha_n^4)$. 
Numerical experiments illustrating this substantial acceleration  can be found in \cite[\S 3.2]{laumey22e}.   

This deterministic analogue of SA, known as quasi-stochastic approximation (QSA), is largely motivated by applications such as reinforcement learning and gradient-free optimization, in which  the algorithm designer also designs $\bfPhi$ for the purpose of ``exploration''.

The goal of this paper is to extend the theory of QSA to algorithms in two timescales,   for which the algorithm objective is to solve a pair of root finding problems $ \barg (\theta^*,\lambda^*) = \barh(\theta^*,\lambda^*) =0$,   for a pair of functions  $\barg,\barh: \Re^{2d} \to \Re^d$.   In the singular perturbation approach to analysis,  it is assumed that $\lambda^* =  \fasttarg(\theta^*)$  for a function $\fasttarg \colon\Re^d\to\Re^d$.

\begin{subequations}

	It is assumed that  $\barg$ and $\barh$ are defined as expectations, expressed here in sample path form
	\begin{align}
\barg(\theta, \lambda) &\eqdef \lim_{T \to \infty} \frac{1}{T} \int^{T}_0 g(\theta, \lambda,\qsaprobe_t) \, dt
\label{e:barg_def}
\\
\barh(\theta, \lambda) &\eqdef \lim_{T \to \infty} \frac{1}{T} \int^{T}_0  h(\theta, \lambda,\qsaprobe_t) \, dt
\label{e:barh_def}
	\end{align}
	The \textit{probing signal} $\{\qsaprobe_t\} \subseteq \Re^m$  is  of the form $\qsaprobe_t = G(\Phi_t)$ in which $\bfPhi$ is the state process for a dynamical system, interpreted as a deterministic   Markov process.   
	Justification for the above limits may be found in \cite{CSRL,laumey22d} under the assumptions of the paper.

	\label{e:bargbarh_def}
\end{subequations}

\begin{subequations}

	The  QSA algorithms considered in this paper are expressed in continuous time,  
	\begin{align}
\ddt \ODEstate_t 
&= a_t g(\ODEstate_t,\Lambda_t, \qsaprobe_t)
\label{e:ODEstate_QSA_Gen}
\\
\ddt \Lambda_t 
& = b_t h(\ODEstate_t,\Lambda_t, \qsaprobe_t)
\label{e:Lambda_QSA_Gen}
	\end{align}
	in which  the  gain processes $\{a_t, b_t\}$ are  non-negative.
	
	\label{e:QSA2_Gen} 
\end{subequations}

Three gain choice settings might be considered:
\wham{1. Constant gain:}       $a_t \equiv \alpha$ and $b_t \equiv \beta$,  with $0<\alpha\ll \beta$.
The pair of ODEs   \eqref{e:meanflowTwoTime} is known as a  singular perturbation model \cite{kokorekha99}.     
The analytical approach of 
\cite{laumey22d}
based on the \textit{perturbative mean flow} extends easily to this setting.

\wham{2. Vanishing gain:}   
Both tend to zero,  and  $b_t/a_t\to\infty$ as $t\to\infty$.  This is favored in actor-critic methods and in some applications to optimization \cite{konbor99,kontsi03a,bhafumarwan03,bhafumarfar01,bhafumarbha01}.

\wham{3. Mixed case:}     $a_t$ is vanishing, but $b_t \equiv \beta>0$ is held fixed.   
One example is extremum seeking control, in which the fast ODE emerges as the state of a high pass filter \cite{liukrs12}.   Another example is policy gradient methods for reinforcement learning, in which $\lambda_t$ is the state process of a dynamical system to be controlled \cite{CSRL}.

For any choice of gain, the pair \eqref{e:QSA2_Gen} has the \textit{two-time scale} property: $a_t$ is small compared to $b_t$.    

\begin{subequations}
	
	Analysis for the first two cases is based on a family of mean flow equations:
	\begin{align}
\ddt \odestate_t  & = \barg(\odestate_t,  \fasttarg(\odestate_t) )
\label{e:meanflowSlow}
\\
\ddt \lambda_t^\theta & = \barf(\theta, \lambda_t^\theta  )
\label{e:meanflowFast}
	\end{align}
	The two ODEs do not interact:
	$\theta\in\Re^d$ is held fixed in the \textit{fast} ODE \eqref{e:meanflowFast},  and 
	the \textit{slow} ODE \eqref{e:meanflowSlow} is autonomous, since $\lambda_t$ is replaced by $ \fasttarg(\odestate_t) $.   
	
	\label{e:meanflowTwoTime}
\end{subequations}

The present paper focuses on the mixed-gain setting because the applications are most compelling in current research, and because the analysis is most  interesting.

As in most papers on two time-scale algorithms, we introduce for the purposes of analysis the  
 family of QSA ODEs parameterized by  $\theta\in\Re^d$:
\begin{equation}
	\ddt \Lambda_t^\theta 
	=  \beta h(\theta ,\Lambda_t^\theta, \qsaprobe_t)
	\label{e:Lambda_QSA_frozen}
\end{equation}
Under the assumptions imposed we establish the existence and uniqueness of a steady-state distribution $\upmu_\theta$ for $(\Lambda_t^\theta, \Phi_t)$.

Analysis of the full QSA ODE combines elements of two classical approaches:

\whamit{1. Singular perturbations.}
A family of models is considered, parameterized by small $\beta>0$.   In this case the goal is to establish  
$\Lambda_t \approx \fasttarg(\ODEstate_t)$ for large $t$, along with error bounds.

\whamit{2. Parameter dependent noise.}
\eqref{e:ODEstate_QSA_Gen} is regarded as a single timescale algorithm, in which the driving noise $(\bfLambda,\bfPhi)$  is parameter dependent.  
Its mean flow is defined by
\begin{equation}
	\ddt \odestate_t = \barg_0(\odestate_t)  \, , \quad 
	\barg_0(\theta) = \int g(\theta,\lambda,G(z))  \upmu_\theta(d\lambda, dz)  
	\label{e:barg0}
\end{equation}

\wham{Contributions} The main contributions of the paper are summarized herein.

\whamrm{(i)} 

\Cref{t:PMF} establishes the
\textit{perturbative mean flow} (p-mean flow) representation for the ``fast'' QSA ODE \eqref{e:Lambda_QSA_Gen}:
\begin{equation}
	\begin{aligned}
\ddt \Lambda_t    
&=  
\beta[ \barh (\ODEstate_t,\Lambda_t)  
-  \beta \barUpupsilonff(\ODEstate_t,\Lambda_t) +  \clW_t]
\\
\clW_t  
&=
\sum^2_{i = 0} \beta^{2-i} \frac{d^i}{dt^i} \clW_t^i 
\label{e:Pmeanflow_lambda}
	\end{aligned}
\end{equation}
in which $\{\clW_t^i ,  \barUpupsilonff_t: i = 0,1,2\}$ are smooth functions of time identified in the theorem.
Moreover, conditions are identified under which $\barUpupsilonff$ is identically zero, which has valuable implications to algorithm design.

The representation in \eqref{e:Pmeanflow_lambda}
invites filtering techniques for error attenuation:    the terms $\{\clW_t^i : i =1,2\}$ are zero-mean and can be attenuated through a second order low-pass filter.

\whamrm{(ii)}  
Estimation error bounds are obtained in \Cref{t:ROC}: for a vector $\theta^\beta \in \Re^d$ satisfying   $\| \theta^\beta - \theta^*  \| = O(\beta) $,
\[
\begin{aligned}
	 \|   \ODEstate_t - \theta^\beta  \| = O(a_t)
	\\
	\limsup_{t \to \infty} \|   \Lambda_t - \fasttarg(\theta^*)  \| = O(\beta)
\end{aligned}
\]

\whamrm{(iii)} 
The introduction of filtering in \Cref{t:ROCaveraging} yields attenuation of the estimation errors:  for a vector $\theta^\beta \in \Re^d$ satisfying   $\| \theta^\beta - \theta^*  \| = O(\beta^2) $ and a sequence $\{ \LambdaF_t\}$ obtained from passing $\{\Lambda_t\}$ through a low-pass filter,
\[
\begin{aligned}
	\|   \ODEstate_t - \theta^\beta  \| = O(a_t)
	\\
	\limsup_{t \to \infty} \|   \LambdaF_t - \fasttarg(\theta^*)  \| = O(\beta^2)
\end{aligned}
\]

\whamrm{(iv)}   Portions of the results given by \Cref{t:ROC} and \Cref{t:ROCaveraging} are based on the justification of the autonomous ODE \eqref{e:barg0} as an approximation to \eqref{e:ODEstate_QSA_Gen}.   
Theory is based on Lyapunov exponents to establish the existence of unique invariant measures $\{ \upmu_\theta : \theta\in\Re^d\}$, and solutions to Poisson's equation for  $\bfPsi^\theta =(\bfLambda^\theta,\bfPhi)$.  Criteria and consequences of a negative Lyapunov exponent are contained in   \Cref{t:fish0}.

\whamrm{(v)}  Examples in \Cref{s:ex} illustrate application of the general theory in (i)--(iv).

Extremum seeking control (ESC) is given as an example of mixed-gain two timescale QSA, for which the mean vector fields $\barg$ and $\barg_0$ are identified. The standard ESC algorithm and the 1SPSA  algorithm of Spall \cite{spa03} are not globally stable in general,  even when gradient descent is stable, because $\barg_0$    is not Lipschitz continuous;  it is pointed out in  \cite{laumey22d}  that we can expect finite escape time when the objective $\Obj$ is a coercive quadratic.  It is shown in  \cite{laumey22d} that Lipschitz continuity and hence global stability can be assured through  the introduction of a state dependent ``exploration gain''.    Extension to the two timescale setting is presented in \Cref{s:ESC}.

Note that in this paper as well as in much of the RL literature, the goal is to estimate a near optimal policy based on training data.   Hence we are not considering the online optimization approach to control, which requires   non-vanishing $\{a_t\}$  (see     
\cite{coldalber20,haubolhugdor21} and their references).

\wham{Literature Survey}

Research in singular perturbation theory  was extremely active within the control systems community in the 1970s, later serving as a foundation for adaptive control. See \cite{smi85} for its century long history and \cite{kokmalsan76,kha02,kokorekha99,sanver07} for  more comprehensive literature surveys on the topic.

Almost sure convergence of $\{\theta_n\}$ to $\theta^*$ for two timescale SA  was established in \cite{bor97a} under the assumption that the sequence of estimates is uniformly bounded almost surely. Bounds on the MSE appeared soon after for the special case of linear SA \cite{kontsi04}. Extensions to non-linear recursions were presented in \cite{mokkpel06}, while criteria for boundedness of estimates 
	 appeared in \cite{lakbha17}. To the best  of our knowledge, the first appearance of two timescale SA with $\bfPhi$ deterministic is the gradient free optimization algorithm in \cite{bhafumarfar01}.

Gradient free optimization methods concern the estimation of $\thetaopt \in \argmin \Obj(\theta)$ based solely on evaluations of the function $\Obj\colon\Re^d\to\Re$, without access to its gradient. A solution based  on stochastic approximation  (SA)  was proposed in the early 1950s by Kiefer and Wolfowitz and refinements  followed over the years \cite{kiewol52,spa03}.
ESC theory followed a parallel development, and in fact was born far before the  introduction of  SA    \cite{EShistory2010,liukrs12}.

The  introduction of tools from the Markov processes literature to QSA  began in \cite[Ch.~4]{CSRL},   and matured significantly in  \cite{laumey22e,laumey22d} following the discovery of conditions to ensure existence of well behaved solutions to Poisson's equation.   This led to the p-mean flow in \cite{laumey22d},  which is a refinement of the noise decomposition of \cite{metpri84,metpri87} based  upon Poisson's equation for Markov chains.

A function analogous to $\barUpupsilonff$ also appears in the p-mean flow for single timescale QSA. It is shown in \cite{laumey22d,laumey23a,laumey23b} that this term is not only a major source of estimation error

in constant gain algorithms, but also may slow down convergence rates when the gain is vanishing.

\wham{Organization} This paper is organized into three additional sections.
\Cref{s:main} covers the notational conventions and assumptions imposed throughout the paper, while also presenting contributions (i)--(iv). \Cref{s:ex} contains examples based upon the mixed gain algorithms that are the focus of this paper. 
Conclusions and directions for future research are contained in \Cref{s:conc}. Proofs and sketches of proofs for some of the main results are given in the Appendix.

	\section{Two Timescale Quasi-Stochastic Approximation}
	\label{s:main}

	\subsection{Assumptions}

	The  $m$-dimensional probing signal is assumed to be of the form 
	$\qsaprobe_t  = G_0(\qsaprobe _t^0  )$
	in which $G_0\colon\Re^K \to\Re^m$ is continuous,  and  
	\begin{equation}
\qsaprobe _t^0  =  [  \cos (2\pi [\,  \omega_1 t  +  \phi_1 ] ) ;\,  \dots  ;\,    \cos (2\pi [\,  \omega_K t  +  \phi_K ] )  ]
\label{e:qsaprobe0}
	\end{equation} 
	Theory requires an alternative representation, and stronger assumptions:  throughout the paper we take 
	$\qsaprobe_t = G(\Phi_t)$, in which $\bfPhi$ is the $K$-dimensional clock process that evolves in a compact set of the Euclidean space denoted $\prstate \subset \Co^K$. 
	It has entries $\Phi_t^i =   \exp(2\pi j[\omega_i t + \phi_i])$ for each $i$ and $t$ and is   the state process for a dynamical system $\ddt\Phi_t = W \Phi_t $, where $W \eqdef  2\pi j\text{\rm diag} (\omega_i)  $.

	The following compact notation is adopted for   \eqref{e:QSA2_Gen}:  
	\begin{align}
\ddt \BIGstate_t &= \Bss_t f(\BIGstate_t,\qsaprobe_t) 
\, , \quad 
\Bss_t = \begin{bmatrix}
	a_t I & 0
	\\
	0 & \beta  I
\end{bmatrix}
\label{e:QSAODE_gen_f}
\\
\barf(x) &\eqdef \lim_{T \to \infty} \frac{1}{T} \int^T_0 f(x,\qsaprobe_t) \, dt
\label{e:barfdef} 
	\end{align}
	We assume that   $\barf(x^*) = 0$ with $x^* = (\theta^*;\fasttarg(\theta^*))$.

	The assumptions on $\bfqsaprobe$ and other assumptions are summarized as follows:

	\wham{(A0i)}     
	$\qsaprobe_t  = G_0(\qsaprobe _t^0  )$ for all $t$,   with $\qsaprobe _t^0$ defined in \eqref{e:qsaprobe0},  and the function 
	$G_0\colon\Re^K \to\Re^m$ is   analytic.

	\wham{(A0ii)}   
	The frequencies $\{\omega_1\,,\dots\,,  \omega_K\}$ are distinct, 
	of the form
	$\omega_i   = \log(a_i/b_i) > 0$ and with   $\{a_i,b_i\}$   distinct positive integers.

	\wham{(A1)}
	$a_t = (1+t)^{-\rho}$ with $\half < \rho < 1$ and $b_t \equiv \beta>0$.  
	\wham{(A2)} 
	The functions $f$ and $\barf$ are Lipschitz continuous:  for a constant $\Lipf <\infty$ and all $x', \, x  \in\Re^{2d}\,,  \qsaprobe, \qsaprobe' \in\Re^m $,
	\begin{align*} 
\|f(x',\qsaprobe) - f(x,\qsaprobe)\|
&\le
\Lipf \|x' - x\| 
\\
\|f(x,\qsaprobe') - f(x,\qsaprobe)\| 
&\le
\Lipf \| \qsaprobe'-\qsaprobe\|
\\
\|\barf(x') - \barf(x)\| 
&\le
\Lipf \|x' - x\|
	\end{align*}

	\wham{(A3)} For each $\theta \in \Re^d$, the ODE $\ddt\lambda^\theta_t  =  \barh(\theta,\lambda_t^\theta)$ has a unique globally asymptotically stable equilibrium $\fasttarg(\theta)$, where $\fasttarg: \Re^d \to \Re^d$ satisfies, for a constant $\Lipfasttarg< \infty$,
	\[
	\| \fasttarg(\theta) -\fasttarg(\theta') \| 
	\leq
	\Lipfasttarg \| \theta -\theta' \|
	\, , \quad \theta,\theta' \in \Re^d 
	\]  
	Moreover, the ODE $\ddt\odestate_t  =  \barg(\odestate_t,\fasttarg(\odestate_t))$ has a unique globally asymptotically stable equilibrium $\theta^*$.

\wham{(A4)}  $\limsup_{t\to\infty} \|\BIGstate_t \|\leq b^\bullet < \infty$.

\wham{(A5)}	The vector fields $f$ and $\barf$ are each twice continuously differentiable.  
Moreover, the matrices    $\barAsstar \eqdef \partial_\theta \barg(\theta^*,\fasttarg(\theta^*))$ and
$\{ \barAff(\theta) \eqdef \partial_{\lambda}\barh(\theta, \fasttarg(\theta)) : \theta \in \Re^d\}$ are assumed Hurwitz.

\wham{(A6)}	There exist functions $V^\circ\colon \Re^d \to \Re_+$ and $V^\bullet\colon \Re^d \to \Re_+$ with bounded gradients satisfying the following bounds for each $\theta, \lambda \in \Re^d$:
\[
\begin{aligned}
	\partial_\lambda V^\bullet(\theta,\lambda) \cdot \barh(\theta,\lambda) 
	&\leq - \| \lambda -  \fasttarg (\theta) \|^2
	\\
	\partial_\theta V^\circ(\theta) \cdot \barg(\theta,\fasttarg(\theta))   &\leq -  \| \theta - \theta^*\|^2
\end{aligned}
\]

Assumptions (A1)--(A6) are variations on standard assumptions in the stochastic approximation literature \cite{bor20a,laumey23a,chedevborkonmey21}. The ultimate boundedness assumption (A4) will follow from the other assumptions.

Assumption (A0) is far from standard.  It is required to obtain error bounds in the main results of the paper.

\wham{Markovian foundations}

The theory in this paper rests on recognition that  $\bfPhi$ is a Markov process.   Some key tools from the theory of Markov processes are firstly ergodicity:    $\bfPhi$ admits a unique invariant measure $\uppi$, 
the uniform distribution on $\prstate$.    These observations justify the law of large numbers (LLN) \eqref{e:bargbarh_def}---see  \cite{CSRL,laumey22d}.

Many of the results of this paper rest on solutions to  Poisson's equation, whose definition depends upon context.   Here we recall the formulation from  \cite{laumey22d}:  Let $\bivstate \eqdef \Re^{2d} \times \prstate$ denote the state space for $(\ODEstate,\Lambda,\Phi)$.  
We denote $\baru(x)  \eqdef \int_\prstate u(x,z) \, \uppi(dz)$  and   $\tilu(x,z) \eqdef u(x,z) - \baru(x)$  for any $(x,z) \in \bivstate$ and  continuous vector-valued function $u $  on $ \bivstate$.      
We say that  $\hau$ is the \textit{solution} to  Poisson's equation 
with \textit{forcing function} $u$ if  
\begin{equation}
	\int^T_0 \tilu(x,\Phi_t) \, dt =    \hau(x , \Phi_0)  -  \hau(x,\Phi_T) 
	\label{e:fish}
\end{equation}
for each $ T\geq 0 $  and $x \in \Re^{2d} $.

Throughout the remaining of the paper, we write $u_t$ instead of $u(\BIGstate_t,\Phi_t)$ for  functions of the larger state process $(\BIGstate_t,\Phi_t)$.

Solutions are not unique, since we may always add a constant to obtain another solution.  Throughout the paper it is assumed that the solution is normalized so that $\int \hau(x,z)\, \uppi(dz) =0$ for each $x$.

\smallskip

If $u$ is smooth in its first variable, then
the directional derivatives  in the directions $g$ and $h$ are denoted
\begin{equation}
	\begin{aligned}
		[ \Dg u ] (x,z )  &= \partial_\theta u\,  (x,z )   \cdot g (x, G(z))
		\\
		[ \Dh u ] (x,z )  &= \partial_\lambda u\,  (x,z )   \cdot h (x, G(z))
		\, , \quad (x,z) \in \bivstate
	\end{aligned}
	\label{e:Dir_deriv}
\end{equation}

\whamit{Assumption A0 and Poisson's equation}  

Assumption (A0) may seem overly restrictive,  but to-date this is the only known condition under which we can establish solutions to Poisson's equation, and bounds on these solutions~\cite{laumey22e,laumey22d}.

A second implication of (A0) concerns the nuisance term $\barUpupsilonff$ appearing in 
\eqref{e:Pmeanflow_lambda}.  Results in \cite{laumey22e,laumey22d} may be extended to the setting of this paper to conclude that  $\barUpupsilonff(x) \equiv 0$ for each $x$ under (A0),  with (A0ii) of particular importance (as counter-examples in this prior work show).

The function $\barUpupsilonff$ is constructed based on the solution to Poisson's equation with forcing function $u\equiv f$.   The solution $\haf$  shares the same smoothness properties as $f$;  its  Jacobian with respect to $x$,  denoted  $\haA(x ,z)  \eqdef  \partial_x \haf\, (x, z)$, is  a uniformly bounded function on $\bivstate$.

Denote $\Upupsilon(x,z) \eqdef   
- \haA   (x,z )    f (x, G(z))$.   This is a function $\Upupsilon\colon \bivstate \to  \colon\Re^{2d}$ 
which admits the representation,
\begin{equation}
	\begin{aligned}
		\Upupsilon(x,z) 
		& =
		\begin{bmatrix}
			\Upupsilonss(x,z) + \Upupsilonfs(x,z)
			\\
			\Upupsilonsf(x,z) + \Upupsilonff(x,z)
		\end{bmatrix} 
		\\
		\text{where }\quad \Upupsilonss&= - [\Dg \hag] \, , \quad \Upupsilonsf= - [\Dg \hah]
		\\
		\Upupsilonfs&= - [\Dh \hag] \, ,  \quad \Upupsilonff= - [\Dh \hah]
	\end{aligned}	
	\label{e:Up_newnotation}
\end{equation}
Denoting $\barUpupsilon (x)  = \int \Upupsilon(x,z)  \, \uppi(dz)$ for $x\in\Re^{2d}$,   which has a similar decomposition,    the 
term $\barUpupsilonff$ appearing in 
\eqref{e:Pmeanflow_lambda}  is precisely $\barUpupsilonff  = \int \Upupsilonff(x,z)  \, \uppi(dz)$.

\subsection{Main Results}

The p-mean flow representation for the fast QSA ODE is described first, in which the terms are defined based on $\haf$ and $\Upupsilon$.   The functions $\Upupsilon$ and $\haf$ are themselves treated as forcing functions in Poisson's equation, with solutions  denoted $\hahaf$ and $\haUpupsilon$, respectively.

\begin{subequations} 
	\begin{theorem}[Perturbative Mean Flow]
		\label[theorem]{t:PMF}
		Suppose that (A0)-(A2) hold.    
		Then, 	the p-mean flow representation \eqref{e:Pmeanflow_lambda} holds with
		$\barUpupsilonff\equiv 0$,  $ \clW_t^2   \eqdef   \hahah_t  $,   
		\begin{align}
			\clW_t^0    
			& \eqdef  -  [\Dh \haUpupsilonff ]_t  
			+
			\frac{a_t}{\beta^2} \Bigl[ r_t  [\Dg \hahah ]_t 
			-  \Upupsilonsf_t  -  
			\beta [\Dg \haUpupsilonff ]_t \Bigr]		
			\label{e:AllTheNoise0}   
			\\
			\clW_t^1    
			& \eqdef  -[\Dh \hahah ]_t   +    \haUpupsilonff_t    
			- \frac{a_t}{\beta}   [\Dg \hahah ]_t   
			\label{e:AllTheNoise1}
		\end{align}  
		\label{e:BigGlobalODE}%
		where $r_t \eqdef \rho/(1+t)$.  
	\end{theorem}

\end{subequations}

\Cref{t:PMF} is one of several steps required in establishing convergence of $\bfODEstate$ and the estimation error bounds in \Cref{t:ROC}.  The bounds that follow  are identical to what is obtained for  single timescale QSA, as well as in the averaging literature \cite{laumey22d,kha02}.     
\begin{theorem}[Convergence]
	\label[theorem]{t:ROC}
	Suppose (A0)--(A6) hold. Then, there exist finite constants $\beta^0, \bdd{t:ROC}$ 
	such that for any $0<\beta \le \beta^0$ there is $\theta^\beta \in \Re^d$ satisfying $\| \theta^* - \theta^\beta  \| \leq \bdd{t:ROC} \beta$, and
	\whamrm{(i)} $ \|  \ODEstate_t - \theta^\beta   \|   \le   \bdd{t:ROC} a_t $ for $t \geq 0$.
	\whamrm{(ii)} $\displaystyle \limsup_{t \to \infty}\|  \Lambda_t - \fasttarg(\theta^*)   \|  \le  \bdd{t:ROC} \beta $.
\end{theorem}

The estimation error in \Cref{t:ROC}~(ii) may be attenuated via filtering, which in turn implies convergence of $\bfODEstate$ to a value closer to $\theta^*$.   
Next, we consider a second order filter for the fast variable:
\begin{equation}
	\tfrac{d^2}{dt^2}  \LambdaF_t 
	+
	2\gamma\zeta\tfrac{d}{dt}  \LambdaF_t
	+\gamma^2\LambdaF_t  = \gamma^2 \Lambda_t \, , 
	\label{e:2nd_order_filter_FAST}
\end{equation} 
and this is used in the slow dynamics,
\begin{equation} 
	\ddt \ODEstate_t   = a_t g(\ODEstate_t, \LambdaF_t, \qsaprobe_t)
	\label{e:ODEstate_QSA_F}
\end{equation} 
In the following we impose the constraint on the natural frequency, $\gamma = O(\beta)$.

\begin{theorem}[Error Attenuation]
	\label[theorem]{t:ROCaveraging}
	Suppose (A0)--(A6) hold,  and the second-order  filter is chosen subject to the following constraints:  
	the damping ratio $\zeta \in (0,1)$ is independent of $\beta$, 
	and a constant  $\eta>0$ is also fixed to define the natural  frequency, $\gamma = \eta \beta$ for each $\beta$.
	Then, there exists $\beta^0$ such that for any $0<\beta \le \beta^0$ there is $\theta^\beta \in \Re^d$ satisfying $\| \theta^* - \theta^\beta  \| \leq \bdd{t:ROCaveraging} \beta^2$, and
	\whamrm{(i)} $ \|  \ODEstate_t - \theta^\beta   \|   \le   \bdd{t:ROCaveraging} a_t $ for $t \geq 0$.
	\whamrm{(ii)} $\displaystyle \limsup_{t \to \infty}\|  \LambdaF_t - \fasttarg(\theta^*)   \|  \le   \bdd{t:ROCaveraging} \beta^2 $.

\end{theorem}

\subsection{Lyapunov exponents and Poisson's equation}
\label{s:rates}
A key step in establishing convergence of $\bfODEstate$ consists of justifying the interpretation of \eqref{e:ODEstate_QSA_Gen} as an instance of single timescale QSA with mean vector field \eqref{e:barg0}. This requires a different formulation of Poisson's equation.    

Consider for each $\theta\in\Re^d$ the ODE \eqref{e:Lambda_QSA_frozen}.
The joint process $\{ \Psi_t^\theta = (\Lambda_t^\theta,\Phi_t) :  t\ge 0\}$ is the state process of a dynamical system (hence a Markov process) for which we might hope to solve Poisson's equation.  

For each $\theta \in \Re^d$, we obtain existence and uniqueness of an invariant measure $\upmu_\theta$ for $\bfPsi^\theta$ in \Cref{t:fish0}, by establishing the existence of a negative Lyapunov exponent  for the system \eqref{e:Lambda_QSA_frozen}. This approach is also used in \cite{laumey22b} to show ergodicity of single timescale QSA with constant gain. The Lyapunov exponent $\clL_\Lambda$ is defined as follows:
\begin{subequations}
	\begin{equation}
	\clL_\Lambda \eqdef \lim_{t \to \infty} \frac{1}{t} \log(\|  \Sens_t \|)
	\label{e:Lyap_exp}
\end{equation}
in which $\{\Sens_t\}$  is known the \textit{sensitivity process} and is defined by
\begin{equation}
	\Sens_t \eqdef \frac{ \partial}{\partial \Lambda_0^\theta} \Lambda_t^\theta 
	\label{e:sensi}
\end{equation}
	\label{e:sensitivity}
\end{subequations}

\Cref{t:fish0} also establishes existence of a solution $\hau_0$ to the following version of Poisson's equation: denoting $\baru_0(\theta)  \eqdef \int u(\theta,\lambda,z) \, \upmu_\theta(d\lambda,dz)$ and   $\tilu_0(x,z) \eqdef u(x,z) - \baru_0(\theta)$  for any $(x,z) \in \bivstate$,
\begin{equation}
	\int^T_0 \tilu_0(\theta,\Psi^\theta_t) \, dt =    \hau_0(\theta , \Psi^\theta_0)  -   \hau_0(\theta , \Psi^\theta_T) \,, \quad T\ge 0\, .  
	\label{e:fish0}
\end{equation}
Once again the solution is assumed normalized, with $\int \hau_0(\theta,\lambda,z) \, \upmu_\theta(d\lambda,dz) =0$ for each $\theta$.

\begin{theorem}
	\label[theorem]{t:fish0} 
	\whamrm{(i)}
	If   (A0)--(A6) hold then there is $\beta_0>0$ 
	and   a continuous function $B\colon\bivstate \to\Re_+$ 
	such that 
	for $0<\beta\le \beta_0$, a constant $\delta>0$ and each initial condition $(\theta,\lambda,z)\in\bivstate$, the following holds:   
	\begin{equation}
		\| \Lambda_t^\theta  -   \Lambda_t^{\theta,0}  \|  \le   B(\theta,\lambda,z)  \exp(-\delta \beta t)
		\label{e:LyapunovExponent}
	\end{equation}
	in which $\Lambda_t^{\theta,0} $ is the solution to \eqref{e:Lambda_QSA_frozen} with initial condition $(\theta,0,z)$.

	\whamrm{(ii)}   
	Suppose that (A0)--(A3) hold, along with \eqref{e:LyapunovExponent},   and the solutions to  \eqref{e:Lambda_QSA_frozen} are bounded in $t$ from each initial condition.
	Then there is a unique invariant measure for $\bfPsi^\theta$ with compact support.
	Moreover, there is a solution $\hau_0$ to \eqref{e:fish0} that is zero mean and locally  Lipschitz continuous,   whenever $u$ is locally  Lipschitz continuous.

\end{theorem}

\def\flow{\phi}

\begin{proof}
	Part (i) follows from the argument in \cite[Section 2]{laumey22b}, based on showing that a Lyapunov exponent is negative for sufficiently small $\beta>0$.
	
	The proof of (ii) is also based on analysis from \cite{laumey22b}:  subject to \eqref{e:LyapunovExponent}, for each $\theta\in\Re^d$ we can apply a technique known as \textit{coupling from the past}   to construct a stationary realization $\{  (\Lambda_t^{\theta,\infty}, \Phi_t^\infty)  :  -\infty < t < \infty \}$.   Its common marginal is $\upmu_\theta$.

	Solutions to Poisson's equation are obtained by examining the construction of this stationary realization:  $\Lambda_t^{\theta,\infty} =   \flow_\infty^\theta (\Phi_t^\infty)$ for each $t$,  for a smooth function $\flow_\infty^\theta$.
	For any initial condition   $\Phi_0 = z$  the process $ \Lambda_t^{\theta,z} =   \flow_\infty^\theta (\Phi_t)$  is a solution to \eqref{e:Lambda_QSA_frozen}, with initial condition $ \Lambda_0^{\theta,z} =   \flow_\infty^\theta (z)\in\Re^d$. 
	
	The construction of $\hau_0$ now proceeds in two steps:  first consider, for each   $\Psi_0^\theta = (\lambda,z)$,
	\[
	\begin{aligned}
		\hau_1(\theta,\lambda,z)  &= \int_0^\infty  \tilu_1(\theta,\Psi_t^\theta)   \, dt \,,   \qquad  \tilu_1   =  \tilu_0 -   \tilu_2
		\\
		\textit{with}
		\quad
		\tilu_2(\Phi_t) &=
		\tilu_0(\theta, \flow_\infty^\theta (\Phi_t),\Phi_t)  
	\end{aligned} 
	\]
	This satisfies Poisson's equation with forcing function $\tilu_1$:   
	$\ddt \hau_1(\theta,\Psi^\theta_t)   =  -  \tilu_1(\theta,\Psi_t^\theta)$ for all $t\ge0$.
	
	Next, let  
	$ \hau_2(\theta,z)$ denote the solution to Poisson's equation in the form  \eqref{e:fish} with forcing function $    \tilu_2  $.         The desired solution to  \eqref{e:fish0} is then $  \hau_0 = \hau_1 - \hau_2$.
\end{proof}

Upon obtaining existence of solutions $\hau_0$ to \eqref{e:fish0}, the interpretation of \eqref{e:ODEstate_QSA_Gen} as an instance of single timescale QSA is then justified based upon the following steps:
\wham{1.} Fixing $\theta \in \Re^d$ and differentiating $\hag_0$ with respect to time yields    
\[
\ddt \hag_0(\theta , \Lambda^\theta_t , \Phi_t)  = \beta \partial_\lambda \hag_0(\theta, \Lambda^\theta_t,\Phi_t) \cdot  h(\theta, \Lambda^\theta_t,\qsaprobe_t) + \partial_\Phi \hag_0(\theta, \Lambda^\theta_t,\Phi_t) \cdot W \Phi_t = - \tilg_0(\theta, \Lambda^\theta_t,\qsaprobe_t)
\]
\wham{2.} The time-derivative when $\bfODEstate$ is time-varying admits a similar expression:   
\[
\begin{aligned}
\ddt \hag_0(\ODEstate_t , \Lambda_t , \Phi_t) &= a_t \partial_\theta \hag_0(\ODEstate_t, \Lambda_t,\Phi_t) \cdot  g(\ODEstate_t, \Lambda_t,\qsaprobe_t) + \beta \partial_\lambda \hag_0(\ODEstate_t, \Lambda_t,\Phi_t) \cdot  h(\ODEstate_t, \Lambda_t,\qsaprobe_t)
+ \partial_\Phi \hag_0(\theta, \Lambda_t,\Phi_t) \cdot W \Phi_t
\\
& =  -\tilg_0(\ODEstate_t, \Lambda_t,\qsaprobe_t)  - a_t \Upupsilon^0_t  \, , \quad \text{in which } \Upupsilon^0_t \eqdef \partial_\theta \hag_0(\ODEstate_t, \Lambda_t,\Phi_t) \cdot  g(\ODEstate_t, \Lambda_t,\qsaprobe_t)
	\end{aligned}
	\]

	\wham{3.}   Step 2 motivates the representation of  \eqref{e:ODEstate_QSA_Gen} as a single timescale QSA ODE:  
	\begin{equation}
	\ddt \ODEstate_t = a_t [\barg_0(\ODEstate_t) + \tilXi^0_t] 
	\, , \quad 
	 \tilXi^0_t \eqdef g(\ODEstate_t, \Lambda_t,\qsaprobe_t) - \barg_0(\ODEstate_t) = -\ddt \hag_0(\ODEstate_t , \Lambda_t , \Phi_t) -a_t\Upupsilon^0_t
	 \label{e:1timescale}
	\end{equation}
	 in which $\hag_0$ and $\Upupsilon^0$ are smooth functions of $(\ODEstate_t , \Lambda_t,\Phi_t)$.
	 
\wham{4.} 
The proofs of	
 \cite[Thms. 4.15 \& 4.24]{CSRL} can be extended to \eqref{e:1timescale} to establish the bound $\| \ODEstate_t - \theta^\beta \| \leq \bdd{t:ROC} a_t$, in which  $\theta^\beta \in \Re^d$ satisfies $\barg_0(\theta^\beta) = 0$.

\section{Examples}
\label{s:ex}

\subsection{Linear model} 
\label{s:LinearQSA}

A general model takes the form 
\begin{equation}
	f(\BIGstate_t) = [\barA + A^\circ_t] \BIGstate_t + B_t
	\label{e:linearSA}
\end{equation}
in which the $(2d)\times (2d)$ matrix valued process $\{ A^\circ_t \}$ has zero mean.   Letting $b\in\Re^{2d}$ denote the mean of $\{B_t\}$,  and assuming $\barA$ is invertible,  we have
$
(\theta^*; \lambda^*) = - \barA^{-1} b
$. We take $b =0$ without loss of generality  throughout the remainder of this section.

A single time-scale QSA algorithm is appropriate if $\barA$ is Hurwitz.   
If not, consider the decomposition into four $d\times d$ blocks: in Matlab notation,  $\barA = [\barAss, \barAsf;  \barAfs ,  \barAff]$.     We then have $ \fasttarg(\theta) = - [\barAff]^{-1}\barAfs \theta$,  and 
\eqref{e:meanflowSlow} becomes the $d$-dimensional ODE,  $\ddt \odestate   =  A^\bullet \odestate$, with $A^\bullet = \barAss - \barAsf [\barAff]^{-1}\barAfs$.
Success of the two timescale algorithm requires that each of the two matrices
$A^\bullet$, $\barAff$ be Hurwitz.

The vector field \eqref{e:linearSA} differs from  the linear model of \cite{kontsi04} because there is multiplicative noise, which significantly complicates analysis: see discussion in \cite{laumey22d}. 

Consider for example the mixed-gain QSA ODE \eqref{e:QSAODE_gen_f} 
with $\barA = [\alpha,\alpha;-2,-1]$, $B_t = [\sin(\omega_1 t),\sin(\omega_2 t)]^\transpose$ and $A_t =  I B_t$.

Hence, $\barf(x) = \barA x$ and $x^* =0$. The matrix $\barA$ is Hurwitz only for $0<\alpha<1$. When $\alpha>1$,
 the benefits of a two timescale algorithm for stabilization become clear: $\fasttarg(\theta) = -2\theta$, giving stable dynamics for the mean flow:
\[
\ddt \odestate_t   = \barg(\odestate_t,  \fasttarg(\odestate_t) )  = - \alpha  \odestate_t
\]

\def\rF{\text{\rm F}}
\def\rG{\text{\rm G}}
\def\rH{\text{\rm H}}
\def\rJ{\text{\rm J}}
\def\cqsaprobe{\check{\qsaprobe}}

\subsection{Extremum-seeking control}
\label{s:ESC}

This approach to  gradient-free optimization begins with the construction of approximate gradients  
$ \{ \tilnabla_t\Obj : t\ge 0 \}$  based on perturbed observations  of the form $\clY_t \eqdef \clY(\ODEstate_t,\qsaprobe_t)= \Obj(\ODEstate_t + \upepsilon_t \qsaprobe_t)$, in which  $\upepsilon_t \equiv \upepsilon(\ODEstate_t) $ is   known as the \textit{probing gain}.  Two possibilities result in a Lischitz QSA algorithm provided $\nabla \Obj$ is Lipschitz continuous:
\begin{equation}
	\begin{aligned}
		\upepsilon(\theta) & = \epsy\sqrt{1+\Obj(\theta)  }   
		\\
		\upepsilon(\theta) &= \epsy\sqrt{1+\| \theta - \theta^\ctr \|^2/{\sigma_p^2}} 
	\end{aligned}
	\label{e:expGain}
\end{equation}
where in the first choice it is assumed without loss of generality that   $\Obj$ takes on non-negative values, $\theta^\ctr$ is an a-priori estimate of $\thetaopt$ and $\sigma_p$ plays the role of the standard deviation around this prior.

The observations are normalized
\begin{equation}
	\clYn_t \eqdef \clYn(\ODEstate_t,\qsaprobe_t) =   \frac{1}{\upepsilon_t} \Obj(\ODEstate_t + \upepsilon_t \qsaprobe_t)
	\label{e:normalized_Obs}
\end{equation}
To obtain $\{\ODEstate_{t} \}$, the probing signal and $\{\clYn_t\}$ are fed as input to a sequence of filters \cite{arikrs03}.

The first is a high-pass filter with state process $\Lambda$:
\begin{equation}
	\begin{aligned}
		\ddt \Lambda_t   &=  \rF \Lambda_t + \rG  u_t \, 
		\\
		y_t        & =   \rH ^\transpose  \Lambda_t + \rJ  u_t  \,
	\end{aligned}
	\label{e:HPss}
\end{equation}
with  $(\rF,\rG ,\rH ,\rJ )$  of compatible dimension.  In this equation, $u_t$ is the scalar input and  $y_t$ the scalar output.

The output of \eqref{e:HPss} is expressed $y_t = [\rM u]_t$ with transfer function $\rM(s) \eqdef \rH ^\transpose(Is - \rF )^{-1} \rG + \rJ$. 
We define $\chclYn_t = [\rM\clYn]_t $ and $\cqsaprobe^i_t = [\rM \qsaprobe^i]_t$ for $1\leq i \leq m$, where $\qsaprobe^i_t$ denotes the $i^{\text{th}}$ component of the probing signal.

The $m+1$ outputs of \eqref{e:HPss} are then fed into a low pass filter with state process $\ODEstate$:
\begin{equation}
	\ddt \ODEstate_t =   - a_t[ \sigma ( \ODEstate_t  -\theta^\ctr )  +   a_t     \cqsaprobe_t \chclYn_t]
	\label{e:ESC_LP}
\end{equation}
with $\theta^\ctr $ as defined in \eqref{e:expGain}.

We arrive at a two timescale QSA ODE of the form \eqref{e:QSAODE_gen_f}:
\begin{equation}
	\begin{aligned}
		f(X_t,\qsaprobe^\circ_t)  &=  
		\begin{bmatrix}    
			-  \sigma  I   &  -     \cqsaprobe_t \rH ^\transpose
			\\
			0   &  \rF   
		\end{bmatrix}  X_t   
		+
		\begin{bmatrix}    
			-  \rJ  \cqsaprobe_t   
			\\
			\rG 
		\end{bmatrix}   \clYn_t  
		\\
		\barf(x) &= \begin{bmatrix}    
			-  \sigma  I   &      0
			\\
			0   &  \rF   
		\end{bmatrix}  x
		+ 
		\begin{bmatrix}    
			-  \rJ \,  \Expect[\cqsaprobe \clYn(\theta,\qsaprobe)   ]   
			\\
			\rG \, \Expect[\clYn(\theta,\qsaprobe) ] 
		\end{bmatrix}    
	\end{aligned} 
	\label{e:ESC=QSA}
\end{equation}
in which $\qsaprobe^\circ_t \eqdef (\qsaprobe_t , \cqsaprobe_t)$ is a $2m$-dimensional probing signal and $ \Lambda_t $ is state process in \eqref{e:HPss} with input $u_t=\clYn_t$. 
The expectations in  \eqref{e:ESC=QSA} are taken over $\prstate$, upon recalling that $\qsaprobe_t = G(\Phi_t)$.

From  \eqref{e:HPss} and \eqref{e:ESC=QSA}, we conclude that $\beta =1$ always. Moreover, the vector field $\barg_0$ associated with \eqref{e:barg0} can be identified:
\begin{align}
	\barg_0(\theta) &= -  (\sigma I \theta +
	\Expect[\cqsaprobe \, H^\transpose \Lambda^\theta +
	J \cqsaprobe \, \clYn(\theta, \qsaprobe)] )
	\label{e:meanflow_g0}
\end{align}
in which $\qsaprobe_t = G(\Phi_t)$ and the expectation is taken in steady state and with  respect to the measure $\upmu_\theta$ for  $\Psi^\theta =(\Lambda^\theta,\Phi)$.  Once we establish the conditions of \Cref{t:ROC},  we conclude convergence: $\| \ODEstate_t - \theta^1 \| = O(a_t)$, in which $\barg_0(\theta^1) =0$ and $\|  \theta^1 - \theta^*\| =O(1)$.

It is not difficult to establish that the Lyapunov exponent is negative: 
for each $\theta \in \Re^d$, the ODE \eqref{e:Lambda_QSA_frozen} is  linear with additive disturbance:  
\begin{equation}
	h(\theta,\Lambda_t^\theta,\qsaprobe_t) = 
	\rF \Lambda_t^\theta + \rG   \clYn(\theta,\qsaprobe_t)
	\label{e:ESC_lambda_theta}
\end{equation}
The sensitivity process \eqref{e:sensi} solves $\ddt \Sens_t  = \rF \Sens_t$ with $\Sens_0 = I$.
The Lyapunov exponent \eqref{e:Lyap_exp} can be identified $\clL_\Lambda = \Real(\lambda_1)$, where $\lambda_1$ is an eigenvalue of $\rF$ with maximal real part. Provided $\rF$ is Hurwitz, $\Real(\lambda_1)<0$ and \Cref{t:fish0} can be used to conclude that $\bfPsi^\theta$ admits an unique invariant measure for each $\theta$ as well as existence of solutions to \eqref{e:fish0}.

It remains to show that the mean flow with vector field $\barg_0$ is globally asymptotically stable.     This requires additional assumptions on the objective.  Here we assume that $\nabla\Obj$ is Lipschitz continuous and that its norm is coercive, so that the ODE $\ddt x = \nabla\Obj(x)$ is ultimately bounded.   The same conclusion holds for 
$\ddt \odestate = \barg_0(\odestate)$ for sufficiently small $\epsy>0$,  and conditions on the filter.   However, this is only possible when using a state-dependent probing gain.

Analysis of the mean flow begins with a Taylor series approximation of $ \clYn(\theta,\qsaprobe)$ around $\theta$:
\begin{equation}
	\clYn(\theta,\qsaprobe_t) = \frac{1}{\upepsilon(\theta)} \Obj(\theta) + \qsaprobe_t \nabla \Obj(\theta) + O(\upepsilon)
	\label{e:Taylor_noise}
\end{equation}
which applied to \eqref{e:meanflow_g0} gives, with  $ M_0= J\Expect[ \cqsaprobe \qsaprobe] $,   
\begin{equation}
	\barg_0(\theta) = -  (\sigma I \theta +
	\Expect[\cqsaprobe \, H^\transpose \Lambda^\theta] 
	+
	M_0 \nabla\Obj(\theta) + O(\upepsilon)  )
	\label{e:premeanflow}
\end{equation}

It remains to obtain a representation for $\Expect[\cqsaprobe \, H^\transpose \Lambda^\theta] $. In view of \eqref{e:HPss}, we have that for each $\theta$, 
\[
\begin{aligned}
	\rH^\transpose \Lambda^\theta_t &= \int_{-\infty}^t e^{\rF(t-\tau)} \rG \clYn(\theta,\qsaprobe_\tau) \, d\tau
	= \gamma_0 \frac{\Obj(\theta)}{\upepsilon(\theta)}  + \cqsaprobe_t^\transpose \nabla \Obj(\theta)
\end{aligned}
\]
where the last inequality follows from substitution of  \eqref{e:Taylor_noise} and $\gamma_0 = -\rH^\transpose F^{-1} G$ denotes the DC gain of \eqref{e:HPss}. This implies $\Expect[\cqsaprobe \, H^\transpose \Lambda^\theta] = \Sigma_{\tiny \cqsaprobe}  \nabla \Obj(\theta)$ with $\Sigma_{\tiny \cqsaprobe} \eqdef \Expect[\cqsaprobe \cqsaprobe^\transpose]$. 

Together with \eqref{e:premeanflow}, we obtain
\begin{equation}
	\barg_0(\theta) = -  (\sigma I \theta 
	+
	M \nabla\Obj(\theta) + O(\upepsilon)  )
	\, , 
	\quad  
	M = \Sigma_{\tiny \cqsaprobe}  + M_0 
	\label{e:ESCmeanflow}
\end{equation}
Under passivity of \eqref{e:HPss} we have $M + M^\transpose>0$.   Consequently,   the mean flow with vector field \eqref{e:meanflow_g0} is ultimately bounded for sufficiently small $\epsy>0$ in either  of the choices of exploration gain \eqref{e:expGain}.
Heuristic arguments in   \cite{laumey22d} lead to the same approximation as in \eqref{e:ESCmeanflow} and the same stability conclusion provided  the high-pass filter \eqref{e:HPss} is passive.

\section{Conclusions}

\label{s:conc}
This paper extends QSA theory to algorithms with two timescales in the mixed gain setting. The p-mean flow was identified for the fast variable $\Lambda$, inviting filtering techniques for error attenuation.
Also, theory justifying the interpretation of two timescale algorithms as instances of single timescale QSA was introduced. Its implications to establishing rates of convergence for $\{\ODEstate_t\}$ were recognized.

There are several open paths for research:

\witem This paper concerns static root finding problems of the form $\barf(x^*) =0$. It would be exciting to investigate extensions of the p-mean flow to tracking problems, that is, when $\barf$ is a function of time so that the root is time-varying $\{x^*_t\}$.

\witem Establishing rates of convergence for two timescale QSA with vanishing gain is still an open problem. We conjecture that it is simple to extend \Cref{t:ROC} to this case, but can we achieve the $O(b_t^4)$ MSE rates in \cite{laumey22e} for two timescales?

\clearpage

\bibliographystyle{abbrv}
\bibliography{strings,markov,q,QSA,bandits,Twotime} 

\def\cprime{$'$}\def\cprime{$'$}
\begin{thebibliography}{10}

\bibitem{arikrs03}
K.~B. Ariyur and M.~Krsti\'c.
\newblock {\em Real Time Optimization by Extremum Seeking Control}.
\newblock John Wiley \& Sons, Inc., New York, NY, 2003.

\bibitem{bhafumarbha01}
S.~Bhatnagar, M.~C. Fu, S.~I. Marcus, and S.~Bhatnagar.
\newblock Two-timescale algorithms for simulation optimization of hidden markov
  models.
\newblock {\em Iie Transactions}, 33(3):245--258, 2001.

\bibitem{bhafumarfar01}
S.~Bhatnagar, M.~C. Fu, S.~I. Marcus, and P.~J. Fard.
\newblock Optimal structured feedback policies for abr flow control using
  two-timescale spsa.
\newblock {\em IEEE/ACM Transactions on Networking}, 9(4):479--491, 2001.

\bibitem{bhafumarwan03}
S.~Bhatnagar, M.~C. Fu, S.~I. Marcus, and I.-J. Wang.
\newblock Two-timescale simultaneous perturbation stochastic approximation
  using deterministic perturbation sequences.
\newblock {\em ACM Transactions on Modeling and Computer Simulation (TOMACS)},
  13(2):180--209, 2003.

\bibitem{chedevborkonmey21}
V.~{Borkar}, S.~{Chen}, A.~{Devraj}, I.~{Kontoyiannis}, and S.~{Meyn}.
\newblock The {ODE} method for asymptotic statistics in stochastic
  approximation and reinforcement learning.
\newblock {\em arXiv e-prints:2110.14427}, pages 1--50, 2021.

\bibitem{bor97a}
V.~S. Borkar.
\newblock Stochastic approximation with two time scales.
\newblock {\em Systems Control Lett.}, 29(5):291--294, 1997.

\bibitem{bor20a}
V.~S. Borkar.
\newblock {\em Stochastic Approximation: A Dynamical Systems Viewpoint}.
\newblock {Hindustan Book Agency}, Delhi, India, 2nd edition, 2021.

\bibitem{bormey00a}
V.~S. Borkar and S.~P. Meyn.
\newblock The {ODE} method for convergence of stochastic approximation and
  reinforcement learning.
\newblock {\em SIAM J. Control Optim.}, 38(2):447--469, 2000.

\bibitem{coldalber20}
M.~Colombino, E.~Dall'Anese, and A.~Bernstein.
\newblock Online optimization as a feedback controller: Stability and tracking.
\newblock {\em Trans. on Control of Network Systems}, 7(1):422--432, 2020.

\bibitem{haubolhugdor21}
A.~Hauswirth, S.~Bolognani, G.~Hug, and F.~D{\"o}rfler.
\newblock Optimization algorithms as robust feedback controllers.
\newblock {\em arXiv preprint arXiv:2103.11329}, 2021.

\bibitem{kha02}
H.~K. Khalil.
\newblock {\em Nonlinear systems}.
\newblock Prentice-Hall, Upper Saddle River, NJ, 3rd edition, 2002.

\bibitem{kiewol52}
J.~Kiefer and J.~Wolfowitz.
\newblock Stochastic estimation of the maximum of a regression function.
\newblock {\em Ann. Math. Statist.}, 23(3):462--466, September 1952.

\bibitem{kokorekha99}
P.~Kokotovi{\'c}, H.~K. Khalil, and J.~O'Reilly.
\newblock {\em Singular Perturbation Methods in Control: Analysis and Design}.
\newblock Society for Industrial and Applied Mathematics, 1999.

\bibitem{kokmalsan76}
P.~Kokotovic, R.~O'Malley, and P.~Sannuti.
\newblock Singular perturbations and order reduction in control theory --- an
  overview.
\newblock {\em Automatica}, 12(2):123--132, 1976.

\bibitem{konbor99}
V.~R. Konda and V.~S. Borkar.
\newblock Actor-critic--type learning algorithms for {Markov} decision
  processes.
\newblock {\em SIAM Journal on control and Optimization}, 38(1):94--123, 1999.

\bibitem{kontsi03a}
V.~R. Konda and J.~N. Tsitsiklis.
\newblock On actor-critic algorithms.
\newblock {\em SIAM J. Control Optim.}, 42(4):1143--1166 (electronic), 2003.

\bibitem{kontsi04}
V.~R. Konda and J.~N. Tsitsiklis.
\newblock Convergence rate of linear two-time-scale stochastic approximation.
\newblock {\em Ann. Appl. Probab.}, 14(2):796--819, 2004.

\bibitem{kusyin97}
H.~J. Kushner and G.~G. Yin.
\newblock {\em Stochastic approximation algorithms and applications}, volume~35
  of {\em Applications of Mathematics (New York)}.
\newblock Springer-Verlag, New York, 1997.

\bibitem{lakbha17}
C.~Lakshminarayanan and S.~Bhatnagar.
\newblock A stability criterion for two timescale stochastic approximation
  schemes.
\newblock {\em Automatica}, 79:108--114, 2017.

\bibitem{laumey22e}
C.~K. Lauand and S.~Meyn.
\newblock Approaching quartic convergence rates for quasi-stochastic
  approximation with application to gradient-free optimization.
\newblock In S.~Koyejo, S.~Mohamed, A.~Agarwal, D.~Belgrave, K.~Cho, and A.~Oh,
  editors, {\em Advances in Neural Information Processing Systems}, volume~35,
  pages 15743--15756. Curran Associates, Inc., 2022.

\bibitem{laumey22b}
C.~K. Lauand and S.~Meyn.
\newblock Markovian foundations for quasi stochastic approximation with
  applications to extremum seeking control.
\newblock {\em arXiv 2207.06371}, 2022.

\bibitem{laumey23a}
C.~K. Lauand and S.~Meyn.
\newblock The curse of memory in stochastic approximation.
\newblock In {\em Proc. IEEE Conference on Decision and Control}, pages
  7803--7809, 2023.

\bibitem{laumey23b}
C.~K. Lauand and S.~Meyn.
\newblock The curse of memory in stochastic approximation: Extended version.
\newblock {\em arXiv 2309.02944}, 2023.

\bibitem{laumey22d}
C.~K. Lauand and S.~Meyn.
\newblock Quasi-stochastic approximation: Design principles with applications
  to extremum seeking control.
\newblock {\em {IEEE} Control Systems Magazine}, 43(5):111--136, Oct 2023.

\bibitem{liukrs12}
S.~Liu and M.~Krstic.
\newblock Introduction to extremum seeking.
\newblock In {\em Stochastic Averaging and Stochastic Extremum Seeking},
  Communications and Control Engineering. Springer, London, 2012.

\bibitem{metpri84}
M.~{M\'etivier} and P.~{Priouret}.
\newblock Applications of a {Kushner and Clark} lemma to general classes of
  stochastic algorithms.
\newblock {\em Trans. on Information Theory}, 30(2):140--151, March 1984.

\bibitem{metpri87}
M.~Metivier and P.~Priouret.
\newblock Theoremes de convergence presque sure pour une classe d'algorithmes
  stochastiques a pas decroissants.
\newblock {\em Prob. Theory Related Fields}, 74:403--428, 1987.

\bibitem{CSRL}
S.~Meyn.
\newblock {\em Control Systems and Reinforcement Learning}.
\newblock {Cambridge University Press}, Cambridge, 2022.

\bibitem{mokkpel06}
A.~Mokkadem and M.~Pelletier.
\newblock Convergence rate and averaging of nonlinear two-time-scale stochastic
  approximation algorithms.
\newblock {\em Ann. Appl. Probab.}, 16:1671--1702, 2006.

\bibitem{moujunwaibarjor20}
W.~{Mou}, C.~{Junchi Li}, M.~J. {Wainwright}, P.~L. {Bartlett}, and M.~I.
  {Jordan}.
\newblock On linear stochastic approximation: Fine-grained {Polyak-Ruppert} and
  non-asymptotic concentration.
\newblock {\em Conference on Learning Theory and arXiv:2004.04719}, pages
  2947--2997, 2020.

\bibitem{bacmou11}
E.~Moulines and F.~R. Bach.
\newblock Non-asymptotic analysis of stochastic approximation algorithms for
  machine learning.
\newblock In {\em Advances in Neural Information Processing Systems 24}, pages
  451--459, 2011.

\bibitem{orvkerprobacluc22}
A.~Orvieto, H.~Kersting, F.~Proske, F.~Bach, and A.~Lucchi.
\newblock Anticorrelated noise injection for improved generalization.
\newblock In K.~Chaudhuri, S.~Jegelka, L.~Song, C.~Szepesvari, G.~Niu, and
  S.~Sabato, editors, {\em Proc. Intl. Conference on Machine Learning}, volume
  162 of {\em Proceedings of Machine Learning Research}, pages 17094--17116.
  PMLR, 17--23 Jul 2022.

\bibitem{robmon51a}
H.~Robbins and S.~Monro.
\newblock A stochastic approximation method.
\newblock {\em Annals of Mathematical Statistics}, 22:400--407, 1951.

\bibitem{sanver07}
J.~A. Sanders, F.~Verhulst, and J.~Murdock.
\newblock {\em Averaging methods in nonlinear dynamical systems}, volume~59.
\newblock Springer, 2007.

\bibitem{smi85}
D.~R. Smith.
\newblock {\em Singular-perturbation theory: an introduction with
  applications}.
\newblock Cambridge University Press, 1985.

\bibitem{spa03}
J.~C. Spall.
\newblock {\em Introduction to stochastic search and optimization: estimation,
  simulation, and control}.
\newblock John Wiley \& Sons, 2003.

\bibitem{EShistory2010}
Y.~Tan, W.~H. Moase, C.~Manzie, D.~Ne{\v{s}}i{\'c}, and I.~Mareels.
\newblock Extremum seeking from 1922 to 2010.
\newblock In {\em Proc. of the 29th Chinese control conference}, pages 14--26.
  IEEE, 2010.

\bibitem{tsi94a}
J.~Tsitsiklis.
\newblock Asynchronous stochastic approximation and {$Q$}-learning.
\newblock {\em Machine Learning}, 16:185--202, 1994.

\bibitem{royThesis98}
B.~Van~Roy.
\newblock {\em Learning and Value Function Approximation in Complex Decision
  Processes}.
\newblock PhD thesis, Massachusetts Institute of Technology, Cambridge, MA,
  1998.
\newblock AAI0599623.

\end{thebibliography}


\appendix

\clearpage

\centerline{\Large \bf Appendix}

\bigskip

\section{P-mean flow}

The ODE \eqref{e:QSAODE_gen_f} can be expressed in terms its mean vector field:
\begin{equation}
	\begin{aligned}
\ddt \ODEstate_t &= \Bss_t[\barf(\BIGstate_t) + \tilXi_t ]
\, , \quad 
\tilXi_t \eqdef f(\BIGstate_t, \qsaprobe) - \barf(\BIGstate_t)
	\end{aligned}
	\label{e:Apparent}
\end{equation}
where $\tilXi_t = 
\begin{bmatrix}
	\tilXi^\theta_t
	\\
	\tilXi^\lambda_t
\end{bmatrix} = 
\begin{bmatrix}
	g(\BIGstate_t, \qsaprobe) - \barg(\BIGstate_t)
	\\
	h(\BIGstate_t, \qsaprobe) - \barh(\BIGstate_t)
\end{bmatrix} $.

The p-mean flow representation \eqref{e:BigGlobalODE} follows from the next four lemmas.

\begin{lemma}
	\label[lemma]{t:PMFstep0}
	Suppose that for each $(x, \Phi_0) \in \bivstate$ and $t \geq 0$,
\[
\ddt \haF(x,\Phi_t) \eqdef -\tilF(x,\qsaprobe_t) = -  F(x,\qsaprobe_t) + \barF(x)
\]
where $\haF \colon \bivstate \to \Re^d$ is $C^1$. 

Then, on writing $\haF_t = \haF(X_t, \Phi_t)$ , $\tilF_t = \tilF(X_t, \qsaprobe_t)$,
\[
\ddt \haF_t = -\tilF_t + a_t [\Dg \haF](X_t, \Phi_t) + \beta  [\Dh \haF](X_t, \Phi_t)
\]
\end{lemma}
\begin{proof}
This follows from the chain rule and the definition of directional derivatives in \eqref{e:Dir_deriv}:
\[
\ddt \haF(X_t,\Phi_t) = 
\tilF(X_t,\qsaprobe_t) + a_t \partial_\theta \haF(X_t, \Phi_t) \cdot g(X_t , \qsaprobe_t) +  \beta \partial_\lambda \haF(X_t, \Phi_t) \cdot h(X_t , \qsaprobe_t) 
\]
\end{proof}

\begin{lemma}
	\label[lemma]{t:PMFstep1}
	Under (A0)-(A2), $\tilXi_t$ admits the representation
	\begin{equation}
\tilXi_t  = -\ddt  \haf_t  -
\begin{bmatrix}
	a_t \Upupsilonss_t + 	\beta \Upupsilonfs_t
	\\
	a_t \Upupsilonsf_t + 	\beta \Upupsilonff_t
\end{bmatrix}
\label{e:step1}
	\end{equation}
in which $\haf$ denotes the solution to Poisson's equation \eqref{e:fish} with forcing function $f$.
\end{lemma}
\begin{proof}
	Differentiating $\haf$ with respect to time, we obtain from \Cref{t:PMFstep0},
	\[
	\tilXi_t  =   -\ddt  \haf(\BIGstate_t,  \Phi_t) 
	+    
	\haA( \BIGstate_t,  \Phi_t) \Bss_t f( \BIGstate_t, \qsaprobe_t)  
	\]
	The conclusion \eqref{e:step1} then follows from the definitions in \eqref{e:QSAODE_gen_f} and  \eqref{e:Up_newnotation}.
\end{proof}

\begin{lemma}
	\label[lemma]{t:PMFstep2}
	Suppose that (A0)-(A2) hold. Then,
	\begin{equation}
\ddt \hah_t   =  -r_t a_t   [\Dg \hahah ]_t   
+ a_t \ddt  [\Dg \hahah ]_t 
+ \beta \ddt  [\Dh \hahah ]_t
-  \tfrac{d^2}{dt^2} \hahah_t 
\label{e:step2}
	\end{equation}
	where $r_t = \rho/(t+1)$.
\end{lemma}
\begin{proof}
	Another application of \Cref{t:PMFstep0} with $\haF = \hahah$ gives
	\[
	\begin{aligned}
	\ddt \hahah_t  & =    
	\partial_x \hahah_t \cdot \Bss_t f_t
	-  \hah_t
=   
	a_t [\Dg \hahah ]_t + \beta \ddt [\Dh \hahah ]_t
	-  \hah_t
	\end{aligned}
	\]
	Differentiating both sides with respect to $t$ once more yields \eqref{e:step2}:
	\[
	\begin{aligned}
\tfrac{d^2}{dt^2} \hahah_t  &= 
\ddt\{ a_t [\Dg \hahah ]_t \} + \beta \ddt [\Dh \hahah ]_t  -  \ddt \hah_t   
\\
&= -r_t a_t [\Dg \hahah ]_t+ a_t \ddt [\Dg \hahah ]_t + \beta \ddt [\Dh \hahah ]_t
-  \ddt \hah_t
	\end{aligned}
	\]
\end{proof}

The final lemma is immediate from \Cref{t:PMFstep0}:

\begin{lemma}
	\label[lemma]{t:PMFstep3}
	Suppose that (A0)-(A2) hold. Then,
	\begin{equation}
\Upupsilonff_t = \barUpupsilonff(\BIGstate_t) + a_t [\Dg \haUpupsilonff]_t + \beta [\Dh\haUpupsilonff]_t - \ddt \haUpupsilonff_t
\label{e:step3}
	\end{equation}
\qed
\end{lemma}

\smallskip

\textit{Proof of \Cref{t:PMF}.} From \eqref{e:step1} we obtain
\[
\tilXi_t^\lambda = - \ddt \hah_t - a_t\Upupsilonsf_t - \beta \Upupsilonff_t
\]
The p-mean flow representation follows from substitution of the results in \Cref{t:PMFstep2,t:PMFstep3}.%
\qed

\section{Estimation Error Bounds}

This section begins with a fact about general ODEs. Consider the system:
	\begin{equation}
\ddt x^\circ_t =  \barf^\circ(x^\circ_t)
\label{e:baseODE}
	\end{equation}
	in which $\barf^\circ: \Re^d \to \Re^d$
	is Lipschitz continuous.

	If \eqref{e:baseODE} is globally exponentially asymptotically stable with equilibrium $x^*$, then there exists a function $V^\circ\colon \Re^d \to \Re_+$ with Lipschitz gradient satisfying, for constants $\delta_1, \delta_2, \delta_3 >0$ and each $x \in \Re^{d}$:
	\begin{subequations}
	\begin{align}
 \delta_1	\| x - x^*\|^2 \leq V^\circ(x) &\leq \delta_2 \| x  - x^* \|^2
 \label{e:ident1}
\\
\nabla V^\circ(x) \cdot \barf^\circ(x) 
&\leq 
-  \delta_3 \|  x  - x^*   \|^2   
	\end{align}
	Taking $V = \sqrt{V^\circ}$, it follows that for a constant $\delta>0, \ddt V(x^\circ_t) \leq -\delta V(x^\circ_t)$. Moreover, $V$ is Lipschitz continuous: for a constant $L_V$ and all $x,x' \in \Re^d$,
	\begin{equation}
|V(x) - V(x')| \leq L_V \|  x - x'  \|
\label{e:LipV}
	\end{equation}
\end{subequations}

The next lemma concerns robustness of perturbed versions of \eqref{e:baseODE} of the following form: 
\begin{equation}
	\ddt x_t = \beta[\barf^\circ(x_t) + w_t] \, , \quad \| w_t \|\leq b^w  +\epsy \|x_t \|
	\label{e:noisyODE}
\end{equation}
with $b^w$ and $\epsy$ positive constants. The following is a variant of the small gain theorem.

\begin{lemma}
	\label[lemma]{t:Robust}
	Suppose that the ODE \eqref{e:baseODE} is globally exponentially asymptotically stable with equilibrium $x^*$.
	Suppose in addition that the bound $ \epsy < \delta $ holds for $\delta$ defined above \eqref{e:LipV} and $\epsy$ in \eqref{e:noisyODE}.
	Then, for each $x_{0} \in \Re^d$, the solution to \eqref{e:noisyODE} satisfies 
	\[
	V(x_t) \leq  \clE_t V(x_0)  + \bdd{t:Robust} (1 - \clE_t) \, , \quad  t>0   
	\]
	in which $\clE_t \eqdef \exp(-  \delta' \beta t)$ and $\bdd{t:Robust} = L_V b^w / \delta'$ with $\delta' = \delta - \epsy$.
\end{lemma}
\begin{proof}
	Assume $x^* = 0$ without loss of generality. 
	In view of \eqref{e:noisyODE} and the bound $\ddt V(x^\circ_t) \leq -\delta V(x^\circ_t)$ below \eqref{e:LipV}, computing the derivative of $V$ with respect to time yields
	\begin{equation}
\begin{aligned}
	\ddt V(x_t) 
	&\leq
	-\beta[ \delta V(x_t) - \nabla V(x_t) (b^w   +  \epsy \|  x_t \|)]
	\\
	& \leq -\beta[  \delta' V(x_t)  - L_V b^w   ]
\end{aligned}
\label{e:ddtV_bound}
	\end{equation}
	where $\delta' \eqdef  \delta -\epsy$. 
	The last inequality follows from Lipschitz continuity of $V$ and and the fact that $V$ is a Lyapunov function for \eqref{e:baseODE}, which implies the upper bound $\| x \| \geq V(x)/L_V$.
	Now let $u_t = \exp(\delta' \beta t)$ for $t>0$.
	The bound \eqref{e:ddtV_bound} implies
	\[
	\ddt(u_t V(x_t)) \leq L_V b^w \beta u_t
	\]
	Integrating both sides of the above inequality from $0$ to $t$, and using the fact that $u_{0} = 1$, results in
	\[
	u_t V(x_t)  - V(x_{0}) \leq 
	L_V b^w  \int^t_{0} \beta  u_r \, dr = \frac{L_V b^w }{\delta'} \int^t_{0} \ddr u_r \, dr
	\]
	which completes the proof with $\bdd{t:Robust} \eqdef L_V b^w / \delta'$.
\end{proof}

\Cref{t:Robust} will be used in establishing the results in \Cref{t:ROC}. Preliminary estimation error bounds on the solutions of \eqref{e:QSAODE_gen_f} are first obtained over finite time intervals of length $T>0$ and later extended to arbitrary $t$.
A sequence of sampling  times $\{T_n\}$ is defined with $T_0=0$ and $T_{n+1} - T_n = T/\beta$ for each $n$. 

All constants in the following depend upon the initial condition $X_0$, but are assumed uniformly bounded on compact sets.

The implications to the fast process $\bfLambda$ are presented first. Denote $\epsy^{\Lambda}_{t} \eqdef \Lambda_t - \fasttarg( \ODEstate_t)$,

\begin{proposition}
	\label[proposition]{t:ROC_lambda}
	Suppose the assumptions of \Cref{t:ROC} hold. Then, there is $\beta^0>0$ such that for any $0<\beta\leq \beta^0$ and $n_0$ satisfying $\beta^2 \geq a_{T_{n_0}}$, the following holds for  each $ \BIGstate_{T_{n_0}} = (\ODEstate_{T_{n_0}};\Lambda_{T_{n_0}}) \in \Re^{2d}$:
	\begin{equation}
\|\epsy^{\Lambda}_{t}  \| \leq  \bdd{t:ROC_lambda} \exp(-\bdde{t:ROC_lambda}  [t- T_{n_0}] ) \| \epsy^{\Lambda}_{T_{n_0}}  \|  + 
\bdd{t:ROC_lambda} \beta \, , \quad t \geq  T_{n_0}
\label{e:exp_contraction}
	\end{equation}
where $\bdd{t:ROC_lambda}$ is a constant depending on $X_0$ but independent of $\beta$. Consequently, there exists a finite time $\clT_0$ depending upon $X_0$ and $\beta$ such that
	\begin{equation} \|  \Lambda_t - \fasttarg(\ODEstate_{t})   \| 
	\leq 
	\bdd{t:ROC_lambda} \beta \, , \quad \text{for all } t\geq \clT_0
	\label{e:preROC_lambda}
\end{equation}
%
\end{proposition}

The first step in the proof of \Cref{t:ROC_lambda} is to justify the claim that the slow process $\bfODEstate$ is seen as ``quasi-static'' by $\bfLambda$ over these finite time intervals.

\begin{lemma}
	\label[lemma]{t:ODEstate_static_t}
	Suppose (A1)--(A4) hold. Then, the following holds for a constant $\bdd{t:ODEstate_static_t}$ depending upon $X_0$ but independent of $\beta$:
	\begin{equation}
\| \ODEstate_t - \ODEstate_{T_n} \| \leq \bdd{t:ODEstate_static_t} T \frac{a_{T_n}}{\beta} \, ,
\quad
T_n < t \leq T_{n+1}
\label{e:bound_thetastatic}
	\end{equation}
\end{lemma}
\begin{proof}
	By continuity of $g$ and the assumption that $ \limsup_{t \to \infty} \|X_t\| \leq b^\bullet$, it follows that for each $n$ and a constant $\bdd{t:ODEstate_static_t}$ depending upon $X_0$, $ 	\sup_{\nu \in [T_n , T_{n+1}]} \|g(\BIGstate_\nu , \qsaprobe_\nu )\| \leq \bdd{t:ODEstate_static_t}$.
	Integrating both sides of \eqref{e:ODEstate_QSA_Gen} from $T_n$ to $t$ and taking norms gives \eqref{e:bound_thetastatic} since $\{a_t\}$ is positive and decreases monotonically in $t$.
\end{proof}

For each $\ODEstate_{T_n} \in \Re^d$, denote the \textit{scaled error} for $\bfLambda$ as
\begin{equation}
	Z_t^{n} \eqdef \frac{1}{\beta} (Y_t- \fasttarg(\ODEstate_{T_n})) 
	\, , 
	\quad
	\text{with }Y_t \eqdef \Lambda_t - \beta \hah_t
	\label{e:scaled_lambda}
\end{equation}

Establishing that for each $n$ and all $t$, $\| Z^n_t \| \leq b^\diamond$ in which $b^\diamond$ is a constant independent of $\beta$, implies $ \|\Lambda_t - \fasttarg(\ODEstate_{T_n}) \| = O(\beta)$ for $t \in [T_n,T_{n+1}]$.
To obtain that, we exploit the fact that $Z_t^n$ and $Y_t$ are themselves state processes for dynamical systems.

\begin{lemma}
	\label[lemma]{t:ODE_Y_lambda}
	Suppose (A0)--(A6) hold. Then, the following holds for a constant $\bdd{t:ODE_Y_lambda}$ depending upon $X_0$ but independent of $\beta$,
	\begin{equation}
\begin{aligned}
	\ddt Y_t &=   \beta [ \barh(\ODEstate_{T_n},Y_t)  + w^Y_t] \, , \quad T_n < t \leq T_{n+1}
	\\
	\text {where }
	w^Y_t &=   a_t [\Dh \hag]_t  +  \beta [\Dh \hah]_t + \clE^Y_t 
\end{aligned}
\label{e:Y_ODE}
	\end{equation}
	and  $ \|\clE^Y_t  \| \leq \bdd{t:ODE_Y_lambda} (T a_{T_n}
	+ \beta^2 )/\beta$.
\end{lemma}
\begin{proof}
	In view of the definition of $Y_t$ in \eqref{e:scaled_lambda}, its derivative with respect to time is given by
	\[
	\begin{aligned}
\ddt Y_t &= \ddt \Lambda_t  - \beta \ddt \hah_t 
\\
&= \beta [\barh(\ODEstate_t,\Lambda_t) + a_t [\Dh \hag]_t  + \beta [\Dh \hah]_t ]
	\end{aligned}
	\]
	in which the last expression follows from \eqref{e:Apparent} and \Cref{t:PMFstep1}.
	
	For $T_n < t \leq T_{n+1}$, Lipschitz continuity of $\barh$ and the triangle inequality gives
	\[
	\begin{aligned}
\barh(\ODEstate_t,\Lambda_t) &=  \barh(\ODEstate_{T_n},Y_t) 
+ \clE^Y_t 
\\
\text{where }
\| \clE^Y_t \| &\leq L_f\|\ODEstate_{t} -  \ODEstate_{T_n}\| 
+ \|   \beta \hah_t     \| \leq
\bdd{t:ODE_Y_lambda} \Big(T \frac{a_{T_n}}{\beta}
+ \beta \Big)
	\end{aligned}
	\]
	in which the last bounds follow from \Cref{t:ODEstate_static_t} along with the assumed boundedness of $\bfmX$ in (A4) and the fact that $\hah$ is continuous in its first variable under (A0). 
\end{proof}

\begin{lemma}
	\label[lemma]{t:ODE_Z_lambda}
	Suppose the assumptions of \Cref{t:ODE_Y_lambda} hold. Then, for any $n \geq n_0$ with $n_0$ satisfying $\beta^2 \geq  a_{T_{n_0}}$, 
	\begin{equation}
\begin{aligned}
	\ddt Z_t^{n} &= \beta [ \barAff(\ODEstate_{T_n})Z_t^{n}  + w^Z_t  ] \, , \quad T_n < t \leq T_{n+1}
	\\
	\text{with } 
	\|w^Z_t \| &\leq \bdd{t:ODE_Z_lambda} + O(\beta^0 \min\{ \| Z_t^n  \| , \|  Z_t^n  \|^2 \})
\end{aligned}
\label{e:Z_ODE}
	\end{equation}
	in which $\barAff$ is defined in (A5) and $\bdd{t:ODE_Z_lambda}$ is a constant depending upon $X_0$ but independent of $\beta$. 
\end{lemma}
\begin{proof}
	We begin by considering \eqref{e:Y_ODE}. The definition \eqref{e:scaled_lambda} gives $Y_t = \beta Z^n_t + \fasttarg(\ODEstate_{T_n})$. In view of this, a Taylor series approximation of $\barh$ around $(\ODEstate_{T_n}; \fasttarg(\ODEstate_{T_n}))$  yields
	\begin{align}
\ddt Y_t &= \beta^2\Big[\frac{1}{\beta} \barh(\ODEstate_{T_n}, Y_t) + \frac{1}{\beta}w^Y_t \Big] 
\label{e:preY_ODELin}
\\
&= \beta^2 [\barAff(\ODEstate_{T_n})  Z_t^{n} + \frac{1}{\beta} w^Y_t + \epsy^Z_t]
\label{e:Y_ODELin}
	\end{align}
	in which $\epsy^Z_t$ is the error in the second order Taylor approximation of $\barh$. Under Lipschitz continuity of $\barh$ and the fact that for all $\beta$, $\beta \leq \beta^0$, this error admits the bound $
	\epsy^Z_t = O(\beta^0 \min\{ \|Z_t^{n}\|, \| Z_t^{n}\|^2  \} )$.
	
	Since $n\geq n_0$, it follows that $\clE^Y$ in \eqref{e:Y_ODE} admits the upper bound: $\| \clE^Y_t\| \leq \bdd{t:ODE_Y_lambda} (T+1)\beta$. Moreover, assumption (A0) implies that $\hag$ and $\hah$ are smooth functions of $\bfX$. This along with the assumption that $\limsup_{t \to \infty} \| X_t\| \leq b^\bullet$ in (A4) implies that $\| w^Y_t \| \leq \bdd{t:ODE_Z_lambda} \beta$ for a constant $\bdd{t:ODE_Z_lambda}$ depending upon $X_0$.
	
	Differentiating $Z_t^{n}$ in \eqref{e:scaled_lambda} with respect to time gives $\ddt Z_t^{n} = \frac{1}{\beta} \ddt Y_t$, completing the proof. 
\end{proof}

Assumptions (A4), (A5) and (A6) imply that the ODE \eqref{e:Z_ODE} is exponentially asymptotically stable when $w^Z \equiv 0$. This motivates an application of \Cref{t:Robust} to obtain boundedness of $\{Z_t^{n}\}$.

\begin{lemma}
	\label[lemma]{t:Facts_ZTn}
	Suppose the assumptions of \Cref{t:ODE_Z_lambda} hold. Then, there exists $\beta^0$ such that for all $0< \beta \leq \beta^0$ and any $n \geq n_0$ with $n_0$ satisfying $\beta^2 \geq  a_{T_{n_0}}$, such that
	\whamrm{(i)} there exists $T>0$ large enough so that the contraction holds:
	\begin{equation}
\|  Z^n_{T_{n+1}}  \|
\leq 
\half \|  Z^n_{T_{n}}  \| + \bdd{t:Facts_ZTn}
\label{e:main_contraction}
	\end{equation}

	\whamrm{(ii)} the following bound holds for each $n$: for a constant $\bdd{t:Facts_ZTn}$, 
	\[
	\|  Z^n_{t} -  Z^n_{T_{n}}   \|	
	\leq 
	\bdd{t:Facts_ZTn} (\|  Z^n_{T_{n}} \|+ 1 )		
	\, , \quad  T_n <  t< T_{n+1}
	\]	
where $\bdd{t:Facts_ZTn}$ is a constant depending on $X_0$ but independent of $\beta$. 	
\end{lemma}
\begin{proof}
	Part (i) follows from a direct application of \Cref{t:Robust} to the ODE in \eqref{e:Z_ODE} with $x_t = Z^n_{t}$ and $w_t =w^Z_t$. Moreover, part (ii) is obtained  from an application of the Bellman-Gr{\"o}nwall lemma.
\end{proof}

\begin{lemma}
	\label[lemma]{t:ROC_lambda_Tplusone}
	Suppose the assumptions of \Cref{t:ODE_Y_lambda} hold. If in addition $Z_{T_{n}}^{n+1} = Z_{T_{n}}^{n}$, there exists $\beta^0$ such that for all $0< \beta \leq \beta^0$ and any $n \geq n_0$ with $n_0$ satisfying $\beta^2 \geq  a_{T_{n_0}}$, the following holds:
	\begin{equation}
\| Z_{T_{n+1}}^{n+1} -    Z_{T_{n+1}}^{n}  \|
\leq 
T^2 (L_f \bdd{t:ODEstate_static_t} + 2) \frac{a_{T_n}}{\beta^2}
	\end{equation}
	in which $\bdd{t:ODEstate_static_t}$ is the constant obtained from \Cref{t:ODEstate_static_t}.
\end{lemma}
\begin{proof}
	From \eqref{e:scaled_lambda}, we have that $\ddt Z_{t}^{n} = \frac{1}{\beta} \ddt Y_t $. In view of this identity, integrating both sides of \eqref{e:preY_ODELin} from $T_n$ to $T_{n+1}$ yields
	\[
	\| Z_{T_{n+1}}^{n+1} -    Z_{T_{n+1}}^{n}  \|
	\leq 
	\int^{T_{n+1}}_{T_n} \Big\| \Delta^n_t  + \frac{T}{\beta}(a_{T_{n+1}} - a_{T_{n}}) \Big\|  \, dt
	\]
	where $\Delta^n_t = \barh(\ODEstate_{T_{n+1}} , Y_t) -  \barh(\ODEstate_{T_{n}} , Y_t)$. 
	Lipschitz continuity of $\barh$ implies the bound
	\[
	\| \Delta^n_t \| \leq L_f \|  \ODEstate_{T_{n+1}} -\ODEstate_{T_{n}}  \| 
	\leq 
	L_f \bdd{t:ODEstate_static_t} T \frac{a_{T_n}}{\beta}
	\]
	where the last inequality follows from \eqref{t:ODEstate_static_t}. Since $\{a_t\}$ decreases monotonically in $t$, 
	we have, via the triangle inequality, 
	\[
	\int^{T_{n+1}}_{T_n} \Big\| \Delta^n_t  + \frac{T}{\beta}(a_{T_{n+1}} - a_{T_{n}}) \Big\|  \, dt
	\leq 
	T^2 (L_f \bdd{t:ODEstate_static_t} + 2) \frac{a_{T_n}}{\beta^2}
	\]
	which completes the proof.
\end{proof}

\smallskip
\textit{Proof of \Cref{t:ROC_lambda}.}
For each $n \geq  n_0$, the contraction in part (i) of \Cref{t:Facts_ZTn} along with the triangle inequality gives
\[
\begin{aligned}
	\| Z^{n+1}_{T_{n+1}} \|
	&\leq 
	\half \| Z^{n+1}_{T_{n}}\| + \bdd{t:Facts_ZTn}
	\\
	&\leq 
	\half \| Z^{n}_{T_{n}}\|+ \| Z^{n+1}_{T_{n}} -  Z^{n}_{T_{n}}   \| + \bdd{t:Facts_ZTn}
\end{aligned}
\]
in which $\| Z^{n+1}_{T_{n}} -  Z^{n}_{T_{n}}   \| \leq T^2 (L_f \bdd{t:ODEstate_static_t} + 2)  $ via  \Cref{t:ROC_lambda_Tplusone} and the fact that $\beta^2 \geq  a_{T_n}$ for all $n \geq  n_0$. 
Repeating this process, we obtain:
\[
\| Z^{n+1}_{T_{n+1}} \| \leq \Big( \frac{1}{2}\Big)^{n - n_0}  \| Z^{n_0}_{T_{n_0}}\| +  b^\triangle  \sum_{k = n_0}^{n}\Big( \frac{1}{2}\Big)^k
\] 
in which $b^\triangle  = T^2 (L_f \bdd{t:ODEstate_static_t} + 2) + \bdd{t:Facts_ZTn}$.

To complete the proof we employ the triangle inequality once more along with the bound in part (ii) of \Cref{t:Facts_ZTn} to conclude
\[
\lim_{n\to\infty}  \sup_{ T_n \leq t \leq T_{n+1} }  \| Z^n_{t} \| 
\leq  
\bdd{t:Facts_ZTn}  (b^\triangle + 1)
\]
which implies \eqref{e:exp_contraction}. 
\qed

We now turn to the implications of \Cref{t:ROC_lambda} to $\bfODEstate$. Assumption (A4) implies the existence of a finite time $\clT_\bullet$ depending upon $X_0$ such that $\| X_t\| \leq b^\bullet$ for all $t \geq \clT_\bullet$. This fact is used in the following:
\begin{lemma}
\label[lemma]{t:ODEstate_rewrite}
Under the assumptions of \Cref{t:ROC}, the ODE \eqref{e:ODEstate_QSA_Gen} can be re-written as:
\begin{equation}
\ddt \ODEstate_t = a_t [g(\ODEstate_t,\fasttarg(\ODEstate_t),\qsaprobe_t ) + \Delta^g_{t}  \beta ]
\, , 
\quad 
t \geq \clT_m
\label{e:ODEstate_QSA_Gen_rewrite}
\end{equation}
in which $ \|\Delta^g_{t}\| \leq \bdd{t:ROC_lambda}$ and $ \clT_m \eqdef \max\{ \clT_0 , \clT_\bullet\} $ with $\clT_0 $ the finite time in \Cref{t:ROC_lambda}.
\qed
\end{lemma}
\begin{proof}
 Adding and subtracting $a_t\barg(\ODEstate_t,\fasttarg(\ODEstate_t))$ to the right hand side of \eqref{e:ODEstate_QSA_Gen} gives \eqref{e:ODEstate_QSA_Gen_rewrite} with $ \Delta^g_{t} =  \frac{1}{\beta}[g(\ODEstate_t,\Lambda_t,\qsaprobe_t) - g(\ODEstate_t,\fasttarg(\ODEstate_t),\qsaprobe_t)]$. Lipschitz continuity of $g$ and \eqref{e:preROC_lambda} completes the proof: for $t \geq \clT_0$,
\[
\|\Delta^g_t\| = \frac{1}{\beta} \| g(\ODEstate_t,\Lambda_t,\qsaprobe_t) -  g(\ODEstate_t,\fasttarg(\ODEstate_t),\qsaprobe_t)\| \leq \frac{1}{\beta} \|\Lambda_t - \fasttarg(\ODEstate_t) \| \leq \bdd{t:ROC_lambda}
\] 
\end{proof}

Similarly to what is done in standard single timescale SA, ultimate boundedness of $\bfODEstate $ is established by comparison of this process with solutions to the \textit{mean flow} ODE \cite{bor20a,CSRL,kusyin97}. Here we consider the mean flow associated with $\bfODEstate$:
\begin{equation}
	\begin{aligned}
\ddt\odestate_t  =  \barf(\odestate_t,\fasttarg(\odestate_t)) \, , \quad \text{for $\odestate_0 = \theta = \Re^d$}
	\end{aligned}
\label{e:meanflow_}
\end{equation}

For analysis, it is convenient to define a new timescale based upon the gain process $\{a_t\}$,
\begin{equation}
	\tau = s(t) \eqdef \int^t_0 a_r \, dr \, , \quad t\geq \clT_m
	\label{e:timeTrans_fast}
\end{equation}
with $\clT_m$ as defined by \eqref{t:ODEstate_rewrite}.
The QSA ODE \eqref{e:ODEstate_QSA_Gen_rewrite} can then be expressed in terms of this new timescale as
\begin{subequations}
	\begin{align}
\tfrac{d}{d\tau}  \haODEstate_\tau 
&=   [g(\haODEstate_\tau ,\fasttarg(\haODEstate_\tau) , \haqsaprobe_\tau  )
+ \haDelta^g_\tau  \beta] \, , \quad \tau \geq s (\clT_m)
\label{e:haODEstate}
	\end{align}
	where $\haODEstate_\tau \eqdef \ODEstate_t \mid_{t = s^{-1}(\tau)}$, $\haqsaprobe_\tau \eqdef \qsaprobe_t \mid_{t={s^{-1}(\tau)}}$ and $\haDelta^g_\tau \eqdef \haDelta^g_t \mid_{t = s^{-1}(\tau)}$.
	
	For each $\tau>0$, let $\{\haodestate_\gamma^\tau \colon \gamma \geq \tau\}$ be solutions to ``re-started'' versions of the mean flow ODE \eqref{e:meanflow_}:
	\begin{align}
\tfrac{d}{d\gamma}  \haodestate_\gamma^\tau
&=   \barg(\haodestate_\gamma^\tau ,\fasttarg(\haodestate_\gamma^\tau)  ) 
\, , \quad  \haodestate_\tau^\tau = \haODEstate_\tau 
	\label{e:meanflow_theta}
%
	\end{align}

\end{subequations}

One step in the proof of  \Cref{t:ROC}~(i) is based upon establishing that the difference between solutions to \eqref{e:haODEstate} and \eqref{e:meanflow_theta} is at most proportional to $\beta$ over finite time intervals of length $T^\theta>0$ as $\tau \to \infty$.

The next lemma establishes a version of the law of large numbers over these fixed time intervals. Its proof follows from the exact same steps of \cite[Prop. 4.28]{CSRL}.
\begin{lemma}
	\label[lemma]{t:g_vary_coupling}
	Suppose (A0)--(A6) hold.
	Then, for any $T^\theta>0$,
%
\[
	\lim_{\tau \to \infty} \sup_{\nu \in [0,T^\theta]}  \Big\| \int^{\tau+\nu}_\tau g(\haODEstate_\gamma, \fasttarg(\haODEstate_\gamma) , \haqsaprobe_\gamma) - \barg(\haODEstate_\gamma , \fasttarg(\haODEstate_\gamma)) \,  d\gamma \Big\| =0
\]	
	\qed
\end{lemma}

Solidarity between the solutions to  \eqref{e:haODEstate} and \eqref{e:meanflow_theta} then follows from \Cref{t:g_vary_coupling}.
\begin{lemma}
	\label[lemma]{t:Couplesol_timevary_theta}
	Under the assumptions of \Cref{t:ROC}, the following holds for any $T^\theta>0$,
	\[
	\lim_{\tau \to \infty} \sup_{\nu \in [0,T^\theta]}	\|\haODEstate_{\tau+\nu} - \haodestate_{\tau+\nu}^\tau  \| \leq \bdd{t:Couplesol_timevary_theta} \beta
	\]
	where $\haODEstate_{\gamma}$ solves \eqref{e:haODEstate} and $\haodestate_{\gamma}^\tau$ solves \eqref{e:meanflow_theta}. 
\end{lemma}
\begin{proof}
	Integrating \eqref{e:haODEstate} and \eqref{e:meanflow_theta} from $\tau$ to $\tau+\nu$, we obtain
	\begin{equation}
\begin{aligned}
	\haODEstate_{\tau+\nu} - \haodestate_{\tau+\nu}^\tau  
	& = 
	\int^{\tau+\nu}_\tau [ g(\haODEstate_\gamma,\fasttarg(\haODEstate_\gamma), \haqsaprobe_\gamma) + \haDelta^g_\gamma \beta - \barg(\haODEstate_\gamma,\fasttarg(\haODEstate_\gamma))] \, d\gamma
	\\
	&+  
	\int^{\tau+\nu}_\tau [ \barg(\haODEstate_\gamma,\fasttarg(\haODEstate_\gamma)) - \barg(\haodestate_\gamma^\tau,\fasttarg(\haodestate_\gamma^\tau)) ] \, d\gamma 
	\label{e:precoupleX}
\end{aligned}
	\end{equation}
	where $\haODEstate_\tau = \haodestate^\tau_\tau$.
	Under the assumed Lipschitz continuity in (A2), the second term is bounded as follows
	\[
	\Big\|\int^{\tau+\nu}_\tau [  \barg(\haODEstate_\gamma,\fasttarg(\haODEstate_\gamma)) - \barg(\haodestate_\gamma^\tau,\fasttarg(\haodestate_\gamma^\tau)) ] \, d\gamma \Big\|
	\le
	L_f \int^{\tau+\nu}_\tau  \| \haODEstate_\gamma  - \haodestate_\tau \| \, d\gamma
	\]
	Taking norms of both sides of \eqref{e:precoupleX} yields
	\[
	\| \haODEstate_{\tau+\nu} - \haodestate_{\tau+\nu}^\tau \|  = 
	\delta^\tau_1 + \beta \delta^\tau_2
	+  
	L_g \int^{\tau+\nu}_\tau  \| \haODEstate_\gamma - \haodestate_{\gamma}^\tau  \| \, d\gamma
	\]
	where 
	\[
	\begin{aligned}
\delta^\tau_1 &\eqdef \sup_{\tau' \geq \tau} \max_{0 \leq \nu \leq T^\theta}	\Big\| \int^{\tau'+\nu}_{\tau'} g(\haODEstate_\gamma,\fasttarg(\haODEstate_\gamma), \haqsaprobe_\gamma) - \barg(\haODEstate_\gamma,\fasttarg(\haODEstate_\gamma)) \,  d\gamma \Big \| 
\\
\delta^\tau_2 &\eqdef  \sup_{\tau' \geq \tau} \max_{0 \leq \nu \leq T^\theta}   \int^{\tau'+\nu}_{\tau'} \| \haDelta^g_\gamma \| 
	\end{aligned}
	\]
	An application of Gr\"onwall's inequality then gives $\| \haODEstate_{\tau+\nu} - \haodestate_{\tau+\nu}^\tau \| \leq e^{L_g \nu} (\delta^\tau_1 + \delta^\tau_2) $.
	
	To complete the proof, it remains to obtain bounds on $\delta^\tau_1, \delta^\tau_2$ as $\tau \to \infty$. An application of \Cref{t:g_vary_coupling} establishes $\lim_{\tau \to \infty}\delta^\tau_1 =0$. The second term is bounded as follows: since $\tau \geq s(T^\theta)$,
	\[
	\begin{aligned}
	\lim_{\tau \to \infty} \delta^\tau_2 = \limsup_{\tau \to \infty} \int^{\tau+ T^\theta}_{\tau} \| \haDelta^g_\gamma \|  d\gamma
	&\leq
	\limsup_{\tau \to \infty}   \bdd{t:ROC_lambda}  \int^{\tau+ T^\theta}_{\tau}    d\gamma \leq 
	 \limsup_{\tau \to \infty}  \bdd{t:ROC_lambda} T^\theta \leq  \bdd{t:Couplesol_timevary_theta}
	\end{aligned}
	\]
\end{proof}

\begin{lemma}
	\label[lemma]{t:limsuptheta}
	Under the assumptions of \Cref{t:ROC}, the following holds:,
	\[
	\limsup_{t \to \infty} 	\|\ODEstate_{t} - \theta^*  \| \leq \bdd{t:limsuptheta} \beta
	\]
\end{lemma}
\begin{proof}
	Under (A4), there is $\clT_\bullet$ depending upon $X_0$ such that $\| X_t\| \leq b^\bullet$ for all $t \geq \clT_\bullet$. Since $\ddt \odestate_t  =  \barg(\odestate_t,\fasttarg(\odestate_t))$ is stable with equilibrium $\theta^*$ by (A3), it follows that for each $\epsy>0$, there exists $\clT_\epsy$ such that 
	\[
	\|  \odestate_{\tau+\gamma}^\tau - \theta^*  \| < \epsy \qquad \text{for all $\gamma \geq \clT_\epsy$ when $\|\odestate_\tau^\tau \|\leq b^\bullet$}
	\]
	By \Cref{t:Couplesol_timevary_theta},
	\[
	\begin{aligned}
		\limsup_{\tau \to \infty} \|  \haODEstate_{\tau+\clT_\epsy} - \theta^* \|  
		&
		\leq 
		\limsup_{\tau \to \infty} \|  \haODEstate_{\tau+\clT_\epsy} - \haodestate_{\tau+\clT_\epsy}^\tau \|  
		+  \limsup_{\tau \to \infty} \|  \haodestate_{\tau+\nu}^\tau - \theta^*  \|
		\\
		&<  \bdd{t:Couplesol_timevary_theta}\beta + \epsy
	\end{aligned}
	\]
	which completes the proof since $\epsy$ is arbitrary.

\end{proof}

With the results of \Cref{t:limsuptheta} and \Cref{t:ROC_lambda} in hands, we are ready to prove \Cref{t:ROC} and \Cref{t:ROCaveraging}.

\smallskip

\textit{Proof of \Cref{t:ROC}}
We begin by establishing part (i).
Under the assumptions of the theorem, the ODE \eqref{e:ODEstate_QSA_Gen} can be regarded as an instance of single timescale QSA as explained in \Cref{s:rates}. The algorithm takes the form \eqref{e:1timescale} in which there is $\theta^\beta \in \Re^d$ such that $\barg_0(\theta^\beta)=0$. Then, \cite[Thms. 4.15 \& 4.24]{CSRL} can be extended to the setting of this paper to establish $\| \ODEstate_t - \theta^\beta \| \leq \bdd{t:ROC} a_t$.

Moreover, \Cref{t:ROC}~(i) along with the result in \Cref{t:limsuptheta} implies $\| \theta^\beta - \theta^* \| \leq \bdd{t:ROC} \beta$.

Part (ii) follows directly from part (i), \Cref{t:ROC_lambda} and the assumed Lipschitz continuity of $\fasttarg$ in (A3):
\[
\begin{aligned}
\limsup_{t \to \infty} \| \Lambda_t - \fasttarg(\theta^*)  \| 
&\leq \limsup_{t \to \infty} \|\Lambda_t - \fasttarg(\ODEstate_t) \| + \limsup_{t \to \infty}\| \fasttarg(\ODEstate_t) - \fasttarg(\theta^*)  \| 
\\
& \leq \bdd{t:ROC_lambda} \beta +  \limsup_{t \to \infty}  \Lipfasttarg \|\ODEstate_t - \theta^* \| \leq (\bdd{t:ROC_lambda} + \Lipfasttarg \bdd{t:Couplesol_timevary_theta}) \beta
\end{aligned}
\] 
\qed

\begin{proof}[Proof of \Cref{t:ROCaveraging}]
	The proof extends arguments in \cite{laumey22d}, beginning with the
	the linearization   $   \barh (x)  =     \barAff(\theta) [\lambda -  \fasttarg(\theta) ]  +  \clE(x)$ for any $x=(\theta,\lambda)$,  with $\| \clE(x) \| \le  L_A \min(\|x\|,\|x\|^2)$ for some $L_A<\infty$.   
	The p-mean flow \eqref{e:Pmeanflow_lambda} gives 
	\[	\ddt \Lambda_t    
	=  
	\beta\big[   \barAff(\ODEstate_t) [\Lambda_t -  \fasttarg(\ODEstate_t) ]  +\clE_t     +  \clW_t \big]
	\]
	where $\clE_t = \clE(\BIGstate_t) $ and we have used $\barUpupsilonff\equiv 0$ under (A0).    
	Consequently,
	\[	
	\ddt \LambdaF_t    
	=  
	\beta\big[   \barAff(\ODEstate_t) \tilLambdaF_t    + \epsy_t + \clEF_t     +  \clWF_t \big]
	\]
	in which $\tilLambdaF_t = \LambdaF_t -  \fasttarg(\ODEstate_t)$, and
	letting $\imp$  denote the impulse response for the low pass filter \eqref{e:2nd_order_filter_FAST},
	\[
	\clEF_t  = \int_0^t \imp_{t-\tau} \clE_\tau\, d\tau\,, \quad       
	\clWF_t   = \int_0^t \imp_{t-\tau} \clW_\tau\, d\tau\,, \quad       
	\]
	The error term $\epsy_t$ depends on $X_0$:   it involves the transient response of the filter, and the vanishing term
	\[
	\int_0^t \imp_{t-\tau}  [ \barAff(\ODEstate_t) - \barAff(\ODEstate_\tau)] \tilLambdaF_\tau  \, d\tau
	\]
	Consequently $\epsy_t$  is vanishing as $t\to\infty$.     We also have
	\begin{equation}
		\begin{aligned}
			\limsup_{t \to \infty} \|   \clEF_t   \| \le  
			\limsup_{t \to \infty} \|   \clE_t   \|  &= O(\beta^2)
			\\
			\limsup_{t \to \infty} \|   \clWF_t   \|  &= O(\beta^2)  
		\end{aligned}
		\label{e:limsupsErrorsF}
	\end{equation}
	The upper bound for $\|   \clE_t   \| $ follows from
	\Cref{t:ROC} and the bound  $\| \clE_t \| \le  L_A \|\BIGstate_t\|^2$.
	
	The bandwidth constraint on the low pass filter implies the bound for   $\|  \clWF_t  \|$.  This establishes (ii),  and (i) easily follows as in the proof of \Cref{t:ROC}.
\end{proof}

\end{document}